\documentclass[12pt]{article}

\usepackage[utf8]{inputenc}

\usepackage{amsfonts}
\usepackage{amssymb}
\usepackage{amsmath}
\usepackage{amsthm}
\usepackage[T2A]{fontenc}

\usepackage[russian]{babel}
\usepackage{amsthm}

\def \le {\leqslant}
\def \ge {\geqslant}

\theoremstyle{plain}

\newtheorem{lem}{Лемма}[section]
\newtheorem{foll}{Следствие}[section]
\newtheorem{theor}{Теорема}[section]
\newtheorem{theorbig}{Теорема}
\newtheorem{theor1}{Theorem}

\newenvironment{solve}{\begin{proof}[Доказательство]}{\end{proof}}
\begin{document}

\centerline{\Large \bf On the derivative of two functions from}
\centerline{\Large \bf Denjoy-Tichy-Uitz family.}
\vspace*{3mm}
\centerline{\Large Dmitry Gayfulin.}
\vspace*{10mm}
\centerline{\bf 1. Introduction}
The family of functions, we investigate in this article, was originally introduced by A.Denjoy~\cite{Denjoy} and later rediscovered by R Tichy and J. Uitz\cite{LambdaMink}. We denote the functions of the family by  $g_{\lambda}(x),$ where $\lambda\in(0,1)$. The definition will be given in the following section.
The most famous function of the family is the Minkiowski question-mark function. As we would see, it corresponds to $\lambda=\frac12$. All functions of the family are continuous, strictly increasing and map the segment $[0,1]$ onto itself. Moreover, they are singular i.e. $\forall  \lambda$~the derivative $g'_{\lambda}(x),$ if exists, can take only two values: $0$ and $+\infty.$ In this paper we consider two functions of the class which correspond to $\lambda$ equals $\frac{\sqrt5-1}2$ or $1-\frac{\sqrt5-1}2.$ The aim of this paper is to prove some theorems about essential conditions on x such that if the condition holds then the derivative 
$g'_{\lambda}(x)$ exists and has determined value. The constants used in our theorems are non-improvable.

Our paper is wirtten in Russian. However Inrtoduction and the formulation of main results (Sections 2,3 below) are written in English.
 
\vspace*{3mm}
\centerline{\bf 2. Definitions, notation and preliminaries}
The function $g_{\lambda}(x)$ where $\lambda\in(0,1)$ is defined as follows:
$$\forall\lambda~g_{\lambda}\biggl(\frac01\biggr)=0, ~g_{\lambda}\biggl(\frac11\biggr)=1.$$ 
Then if $g_{\lambda}(x)$ is defined for two consecutive Farey fractions $\frac pq<\frac rs$ we put
$$g_{\lambda}\biggl(\frac{p+q}{r+s}\biggr) =(1-\lambda)g_{\lambda}\biggl(\frac pq\biggr)+\lambda g_{\lambda}\biggl(\frac rs\biggr).$$
For irrational $x\in[0,1]$ the function $g_{\lambda}(x)$ is defined by continious arguments.

Let $x$ be represented as a regular continued fraction $x=[0;a_1,a_2,\ldots, a_t,\ldots].$ One can easily deduce the following identity
\begin{equation}
\begin{gathered}
g_{\lambda}(x)={\lambda}^{a_1-1}-{\lambda}^{a_1-1}(1-\lambda)^{a_2}+{\lambda}^{a_1+a_3-1}(1-\lambda)^{a_2}-\ldots\\+{\lambda^{a_1+a_3+\ldots+a_{2t-1}-1}(1-\lambda)^{a_2+a_4+\ldots+a_{2t-2}}}-\\-\lambda^{a_1+a_3+\ldots+a_{2t-1}-1}(1-\lambda)^{a_2+a_4+\ldots+a_{2t}}+\ldots
\end{gathered}
\end{equation}
This cam be proved by induction on length of the continued fraction. Particulary, for $\lambda={\varphi}^{-1}$ where $\varphi=\frac{\sqrt5+1}2$ is a positive root of the equation $x^2=x+1$ we have
$$
g_{\varphi^{-1}}(x)=\frac1{\varphi^{a_1-1}}-\frac1{\varphi^{a_1+2a_2-1}}+\frac1{\varphi^{a_1+2a_2+a_3-1}}-\frac1{\varphi^{a_1+2a_2+a_3+2a_4-1}}\ldots
$$
\\and for $\tau=1-\frac1{\varphi}$\\
$$
g_{\tau}(x)=\frac1{\varphi^{2a_1-2}}-\frac1{\varphi^{2a_1+a_2-2}}+\frac1{\varphi^{2a_1+a_2+2a_3-2}}-\frac1{\varphi^{2a_1+a_2+2a_3+a_4-2}}\ldots
$$
Denote by $S^{\varphi}_t(x)$ and  $S^{\tau}_t(x)$ the series
\begin{equation}
\begin{gathered}
{S^{\varphi}_t(x)=a_1+2a_2+a_3+2a_4+\ldots=\sum\limits_{i=1}^t a_i(\frac32+\frac12(-1)^i),}\\
{S^{\tau}_t(x)=2a_1+a_2+2a_3+a_4+\ldots=\sum\limits_{i=1}^t a_i(\frac32+\frac12(-1)^{i-1}).}
\end{gathered}
\end{equation}
\vspace*{3mm}
where $a_i$ are partial quotients of continued fraction representation of $x$.
\centerline{\bf 3. Main results}
\begin{theor1}
\label{res1}\quad\\
We put $\varkappa_1=4.$\\
$(\textbf{i})$1. Let for real irrational $x=[0;a_1,...,a_t,...] $ the following
inequality be valid:
$$
\limsup\limits_{t\to\infty}\frac{S^{\varphi}_{2t}(x)}{t}=\varkappa_{sup}(x)<\varkappa_1.
$$
Then the derivative $g'_{\varphi^{-1}}(x)$ exists and $g'_{\varphi^{-1}}(x)=+\infty $\\
$(\textbf{ii})$ For any positive $\delta$ there exists an irrational\\ ${y=[b_1,\ldots,b_t,\ldots]}$ such that:
$$
\lim\limits_{t\to\infty}\frac{S^{\varphi}_{2t}(y)}{t}<\varkappa_1+\delta
~\text{and}~ g'_{\varphi^{-1}}(y)=0.
$$
\end{theor1}
\begin{theor1}
\label{res2}\quad\\
There exists an effectively computable constant $\varkappa_2=13,06+$ such that, if for some real irrational $x=[0;a_1,...,a_t,...] $ the following inequality is valid:
$$
\liminf\limits_{t\to\infty}\frac{S^{\varphi}_{2t}(x)}{t}=\varkappa_{inf}(x)>\varkappa_2
$$
Then the derivative $g'_{\varphi^{-1}}(x)$ exists and equals $0$.\\
$(\textbf{ii})$~ For any positive $\delta$ there exists an irrational\\ ${y=[b_1,\ldots, b_t,\ldots]}$ such that:
$$
\lim\limits_{t\to\infty}\frac{S^{\varphi}_{2t}(y)}{t}>\varkappa_2-\delta~\text{and}~ g'_{\varphi^{-1}}(y)=+\infty.
$$
\end{theor1}
The second statements of the Theorems \ref{res1} and \ref{res2} show that the constants $\varkappa_1$ and $\varkappa_2$ are non-improvable.
\begin{theor1}\quad\\
$(\textbf{i})$~If all partial quotients of $x\in[0,1]\smallsetminus\mathbb{Q}$ are bounded by 2, then the derivative $g'_{\varphi^{-1}}(x)$ exists and equals $+\infty$\\
$(\textbf{ii})$ There exists a quadratic  irrationality $y\in[0,1]$ such that all l partial quotients of $y$ are bounded by $3$ and $g'_{\varphi^{-1}}(y)=0.$
\end{theor1}
\begin{theor1}
\label{res4}\quad\\
Theorems \ref{res1}, \ref{res2} and \ref{res3} hold for the derivative of the function $g_{\tau}(x)$ with the same constants $\varkappa_1, \varkappa_2$ and $S^{\varphi}_t(x)$ substituted by $S^{\tau}_t(x)$ in all statements.
\end{theor1}
\newpage
\centerline{\Large \bf Производные двух функций семейства}
\centerline{\Large \bf Денжуа-Тихого-Уитца.}
\vspace*{3mm}
\centerline{\Large Гайфулин Д. Р.}
\vspace*{8mm}
\centerline{\bf Аннотация}
Семейство сингулярных функций $g_{\lambda}(x),$ где $\lambda\in(0,1)$ было впервые рассмотрено Денжуа в 1938 году и переоткрыто Тихим и Уитцем в 1995 году. Самым известным представителем данного класса является функция Минковского $?(x),$ соответствующая значению $\lambda=\frac12.$ Для сингулярных функций большой интерес представляет вопрос поиска условий на число $x$, при которых можно заведомо сказать, что $g'_{\lambda}(x)=0$ или же $g'_{\lambda}(x)=\infty.$ Для функции Минковского данная задача была впервые рассмотрена в 2001 году Д.Парадизом, П.Виадером и Л.Бибилони и была в основном решена в 2008 году в работе Н.Г.Мощевитина, А.А.Душистовой и И.Д.Кана. В настоящей работе впервые исследуются производные функций $g_{\lambda}(x)$ для значений параметра $\lambda,$ равных $\frac{\sqrt5-1}2$ и $1-\frac{\sqrt5-1}2.$ Констатнты, полученные в работе, являются неулучшаемыми.

\textbf{Библиография}: 12 названий.

\textbf{Ключевые слова}: цепная дробь, континуант, функция Минковского.
\section{Введение}
Функция Минковского $?(x)$ была впервые рассмотрена Германом Минковским в 1904 году~\cite{Hmin}. Она непрерывно, монотонно и взаимно однозначно отображает отрезок $[0,1]$ на себя и определяется на множестве рациональных чисел следующим индуктивным образом: $?(\frac01)=0, ?(\frac11)=1.$ Далее, если для несократимых дробей $\frac pq<\frac rs$ определены значения $?(\frac{p}q)$ и $?(\frac{r}s),$ но $?(x)$ еще не определено ни в одной точке интервала $(\frac{p}q, \frac{r}s),$ то полагают $?(\frac pq \oplus\frac rs)=\frac{?(\frac pq)+?(\frac rs)}2,$ где операция $\oplus$ означает взятие медианты двух дробей: $\frac pq \oplus\frac rs=\frac{p+r}{q+s}.$ Для иррациональных чисел отрезка $[0,1]$ функция определяется по непрерывности.

Пусть
\begin{equation}
x=[0;a_1, a_2, \ldots, a_n, \ldots]=\cfrac{1}{a_1+\cfrac{1}{a_2+{\ldots}}}
\end{equation}
 - разложение в обыкновенную цепную дробь числа $x,$ в этом случае $a_i$ называются неполными частными данной цепной дроби. Будем вместо $[0;a_1, a_2, \ldots, a_n, \ldots]$ использовать более краткое обозначение\\ $[a_1, a_2, \ldots, a_n, \ldots].$ Верна следующая формула, выражающая значение функции Минковского в точке $x$ через неполные частные разложения числа $x$ в цепную дробь~\cite{Salem}: 
\begin{equation}
?(x)=\frac1{2^{a_1-1}}-\frac1{2^{a_1+a_2-1}}+\frac1{2^{a_1+a_2+a_3-1}}-\ldots+\frac{(-1)^{n-1}}{2^{a_1+a_2+\ldots+a_n-1}}+\ldots
\end{equation}
Известно также, что функция Минковского является предельной функцией распределения для значения обыкновенных конечных цепных дробей \cite{MDK}.
Отметим, что функция Минковского и похожие объекты возникает в различных задачах,
связанных с распределением рациональных точек, в частности, в работе Е.П. Голубевой \cite{G},
посвященной построению плоской выпуклой кривой с боьшим количеством целых точек.

В работах~\cite{Denjoy}~и~\cite{LambdaMink}~рассмотрен более общий класс функций, называемых функциями Денжуа-Тихого-Уитца $g_{\lambda}(x), ~\text{где}~\lambda\in(0,1)$. Они имееют такую же область определения и те же значения на концах отрезка $[0,1]:$ $g_{\lambda}(0)=?(0)=0,~g_{\lambda}(1)=?(1)=1,$ но значение в медианте определяется по более общему правилу
$$
g_{\lambda}(\frac pq \oplus\frac rs)=(1-\lambda)g_{\lambda}(\frac pq)+\lambda g_{\lambda}(\frac rs).
$$
Для $g_{\lambda}(x)$ в работе \cite{Zha} доказана аналогичная формула выражения через неполные частные числа $x=[a_1, a_2, \ldots, a_n, \ldots]$, имеющая вид
\begin{equation}
\begin{gathered}
\label{1}
g_{\lambda}(x)={\lambda}^{a_1-1}-{\lambda}^{a_1-1}(1-\lambda)^{a_2}+{\lambda}^{a_1+a_3-1}(1-\lambda)^{a_2}-\ldots+\\+{\lambda^{a_1+a_3+\ldots+a_{2t-1}-1}(1-\lambda)^{a_2+a_4+\ldots+a_{2t-2}}}-\lambda^{a_1+a_3+\ldots+a_{2t-1}-1}(1-\lambda)^{a_2+a_4+\ldots+a_{2t}}+\ldots
\end{gathered}
\end{equation}
Очевидно,что $g_{\frac12}(x)=?(x).$ В работе~\cite{LambdaMink} также показано, что все функции данного класса монотонны, непрерывны и сингулярны (то есть их производные могут принимать только значения $0$ или $+\infty$). 

Обозначим $\varphi=\frac{1+\sqrt{5}}2\approx1.618033989\ldots$ - положительный корень уравнения $x^2-x-1=0$. Эта величина традиционно называется "золотым сечением". В работе Е.Н. Жабицкой~\cite{Zha} было показано, что функция $g_{\lambda}(x)$  при $\lambda={(\frac{\sqrt5-1}2)}^2=\frac1{\varphi^2}\approx0.381966011$ является предельной функцией распределения для так называемых приведенных регулярных цепных дробей, т.е. дробей вида 
\begin{equation}
x =1 - \cfrac{1}{b_1-\cfrac{1}{b_{2} 
- \ldots - \cfrac{1}{b_l }}}, \quad b_i \geqslant 2,
\end{equation}
Обозначим $\tau=\frac1{\varphi^2}.$  Поскольку $1-\frac1{\varphi^2}=\frac1{\varphi},$ для $g_{\tau}(x)$ ввиду (\ref{1}) верна следующая формула: 
\begin{equation}
\begin{gathered}
\label{main}
g_{\tau}(x)=\frac1{\varphi^{2a_1-2}}-\frac1{\varphi^{2a_1+a_2-2}}+\frac1{\varphi^{2a_1+a_2+2a_3-2}}-\ldots
\end{gathered}
\end{equation}
Обозначим 
$$
S^{\tau}_t(x)=2a_1+a_2+2a_3+a_4+\ldots=\sum\limits_{i=1}^t a_i(\frac32+\frac12(-1)^{i-1}).
$$
Тогда формула (\ref{main}) для иррационального $x$ примет вид
\begin{equation}
\label{small}
g_{\tau}(x)=\sum\limits_{i=1}^{\infty}\frac{(-1)^{i-1}}{\varphi^{S^{\tau}_i(x)-2}}
\end{equation}
Введем также сумму 
$$
{S^{\varphi}_t(x)=a_1+2a_2+a_3+2a_4+\ldots=\sum\limits_{i=1}^t a_i(\frac32+\frac12(-1)^i).}
$$
 Поскольку  $\varphi$ удовлетворяет соотношению $\varphi^2=1-\varphi,$ формула~(\ref{1}) принимает для значения параметра $\lambda=\frac1{\varphi}$  вид 
\begin{equation}
\begin{gathered}
\label{main2}
g_{\varphi^{-1}}(x)=\frac1{\varphi^{a_1-1}}-\frac1{\varphi^{a_1+2a_2-1}}+\frac1{\varphi^{a_1+2a_2+a_3-1}}-\ldots
\end{gathered}
\end{equation}
или, с учетом введеных обозначений,
\begin{equation}
\label{phiform}
g_{\varphi^{-1}}(x)=\sum\limits_{i=1}^{\infty}\frac{(-1)^{i-1}}{\varphi^{S^{\varphi}_i(x)-1}}.
\end{equation}
Нетрудно видеть, что
$$
g_{\varphi^{-1}}(x) +g_{\tau}(1-x) =1.
$$
Отсюда следует, что $g_{\varphi^{-1}}(x)$
является аналогичной предельной функцией распределения для  цепных дробей вида
$$
x =\cfrac{1}{b_1-\cfrac{1}{b_{2} 
- \ldots - \cfrac{1}{b_l }}}, \quad b_i \geqslant 2.
$$

\section{Краткая история вопроса}
Впервые задача поиска условия на неполные частные $x,$ при которых $?'(x)$ равно нулю или $+\infty$ была поставлена в работе \cite{Parad} в 2001 году. В ней было выяснено, что ключевым критерием является предельное значение среднего арифметического неполных частных $$S(x)=\lim\limits_{t\to\infty}\frac{a_1+a_2+\ldots+a_t}t.$$ Были найдены константы $\kappa_1$ и $\kappa_2$ такие, что при 
\begin{equation}
\label{amin}
S(x)<\kappa_1~?'(x)=+\infty,~\text{а при}~S(x)>\kappa_2~ ?'(x)=0,
\end{equation}
если соответствующие производные существуют.  В 2007 году в работе \cite{MD} были найдены неулучшаемые значения констант $\kappa_1$ и $\kappa_2$ и доказано, что производная всегда существует при выполнении неравенств (\ref{amin}). В статье \cite{MDK} были также посчитаны неулучшаемые асимптотики функции $S(x),$ уточняющие критерий (\ref{amin}).

Необходимые или достаточные 
условия 
для обращения в ноль
или бесконечность производных функций $g_{\lambda}(x)$ при $\lambda\ne\frac12$
ранее не исследовались. Эта задача 
обсуждается, например, в обзоре Н.Г. Мощевитина \cite{MM}.
Она является существенно более сложной, потому что при $\lambda\ne\frac12$
вклад неполных частных $x$  с четными и нечетными номерами различен (см. формулу (\ref{1})).

\section{Благодарности}
Автор благодарит Н.Г. Мощевитина за постановку задачи и И.Д. Кана за полезные консультации при решении данной задачи и, в особенности, за многочисленные замечания по тексту в ходе написания настоящей статьи.
\section{Формулировка основных результатов}

Сформулмруем наши результаты, касающиеся случаев   $\lambda = \varphi^{-1}$ и $\lambda = \tau=\varphi^{-2} $ .
\begin{theorbig}
\label{res1}\quad\
Положим $\varkappa_1=4.$\\
$(\textbf{i})$ Пусть $x=[a_1,\ldots, a_t,\ldots]$ - иррациональное число и
$$
\limsup\limits_{t\to\infty}\frac{S^{\varphi}_{2t}(x)}{t}=\varkappa_{sup}(x)<\varkappa_1.
$$
Тогда производная $g'_{\varphi^{-1}}(x)$ существует и равна $+\infty$.\\
$(\textbf{ii})$ Для любого положительного $\delta$ существует иррациональное\\ ${y=[b_1,\ldots,b_t,\ldots]}$ такое, что:
$$
\lim\limits_{t\to\infty}\frac{S^{\varphi}_{2t}(y)}{t}<\varkappa_1+\delta
~\text{и}~ g'_{\varphi^{-1}}(y)=0.
$$
\end{theorbig}
\begin{theorbig}
\label{res2}\quad\\
Существует алгоритмически вычислимая с любой точностью константа $\varkappa_2\approx13.05\ldots$ такая, что выполнены следующие утверждения:\\
$(\textbf{i})$~Пусть $x=[a_1,\ldots, a_t,\ldots]$ - иррациональное число и
$$
\liminf\limits_{t\to\infty}\frac{S^{\varphi}_{2t}(x)}{t}=\varkappa_{inf}(x)>\varkappa_2\approx 13.05.
$$
Тогда производная $g'_{\varphi^{-1}}(x)$ существет и равна $0$.\\
$(\textbf{ii})$~Для любого положительного $\delta$ существует иррациональное\\ ${y=[b_1,\ldots, b_t,\ldots]}$ такое, что:
$$
\lim\limits_{t\to\infty}\frac{S^{\varphi}_{2t}(y)}{t}>\varkappa_2-\delta~\text{и}~ g'_{\varphi^{-1}}(y)=+\infty.
$$
\end{theorbig}
Вторые утверждения в теоремах \ref{res1} и \ref{res2} показывают, что константы $\varkappa_1$ и $\varkappa_2$ являются неулучшаемыми.
\begin{theorbig}
\label{res3}\quad\\
$(\textbf{i})$~Если все неполные частные числа $x\in[0,1]\smallsetminus\mathbb{Q}$ меньше либо равны 2, то $g'_{\varphi^{-1}}(x)=+\infty$\\
$(\textbf{ii})$ Существует $y\in[0,1]\smallsetminus\mathbb{Q}$ такое, что все неполные частные числа $y$ меньше либо равны $3$ и $g'_{\varphi^{-1}}(y)=0.$
\end{theorbig}
\begin{theorbig}
\label{res4}
Для производной функции $g_{\tau}(x)$ теоремы \ref{res1}, \ref{res2} и \ref{res3} верны с теми же самыми константами $\varkappa_1, \varkappa_2,$ при этом $S^{\varphi}_t(x)$ заменяется во всех формулировках на $S^{\tau}_t(x).$
\end{theorbig}
\section{Континуанты и цепные дроби. Леммы о производных}
Будем обозначать большими буквами $A, B$ последовательности натуральных чисел произвольной длины: $(a_1, a_2, \ldots ,a_n), (b_1, b_2, \ldots ,b_m),$ определим для них аналогично $S^{\tau}_n(A)=S^{\tau}_n([A]), S^{\varphi}_m(B)= S^{\varphi}_m([B]).$

Через $\langle A\rangle$ обозначается \textit{континуант} - функция от произвольного (возможно пустого) конечного набора натуральных чисел, определенная по следующему правилу:\\
$\langle~\rangle=1,$\\
$\langle a_1\rangle=a_1.$\\
Далее, для $n\geqslant2$ значение континуанта выражается рекуррентно:
\begin{equation}
\langle a_1,a_2,\ldots,a_{n-1},a_n\rangle=a_n\langle a_1,a_2,\ldots,a_{n-1}\rangle+\langle a_1,a_2,\ldots,a_{n-2}\rangle.
\end{equation}
Числа $a_i$ называются \textit{неполными частными} континуанта или соответствующей ему цепной дроби $[A]$. При этом там, где не это будет вызывать путаницы, мы будем понимать под континуантом как саму последовательность $(a_1, a_2,\ldots,a_{n-1},a_n),$ так и значение континуанта от нее.
Через $ { A}$  всегда будем обозначать последовательность неполных частных, то есть $(a_1, a_2, \ldots ,a_n)$, а  через $\overleftarrow{A}$ - в обратном порядке т.е. $(a_n, a_{n-1}, \ldots ,a_1)$.
Иногда   вместо $A$ мы будем использовать обозначение
$\overrightarrow{ A}$, чтобы подчеркнуть различие между прямым и обратным порядком.
 Обозначим также $A_-=(a_2, a_3, \ldots ,a_n)$ и 
$A^-$=$(a_1, a_2, \ldots ,a_{n-1})$, тогда выполнены соотношения:
\begin{equation}
\begin{split}
\label{A1}
[\overrightarrow{ A}]=[a_1,a_2,\ldots,a_{n-1},a_{n}]=\frac{\langle A_-\rangle}{\langle A\rangle}\\ [\overleftarrow{ A}]=[a_n,a_{n-1},\ldots,a_2,a_1]=\frac{\langle A^-\rangle}{\langle A\rangle}
\end{split}
\end{equation}
Известно следующее свойство (см, например, \cite{Knuth}):
\begin{equation}
\label{knu}
\langle X,Y\rangle=\langle X\rangle\langle Y\rangle+\langle X^-\rangle\langle Y_-\rangle=\langle X\rangle\langle Y\rangle(1+[\overleftarrow{X}][\overrightarrow{Y}]).
\end{equation}
Будем обозначать через $p_i(x)$ и $q_i(x)$  соответственно числители и знаменатели $i-$х подходящих дробей к $x\in[0,1],$ то есть $\frac{p_i(x)}{q_i(x)}=[0;a_1,a_2,\ldots,a_{i-1},a_i].$ В тех случаях, когда это не вызовет путаницы, мы будем опускать аргумент $x$ и писать просто $\frac{p_i}{q_i}.$

Несложно показать, что производная $g'_{\varphi^{-1}}(x)=g'_{\tau}(x)=0$ при $x\in\mathbb{Q}.$ В дальнейшем в данной работе мы будем исследовать  $g'_{\varphi^{-1}}(x)$ и $g'_{\tau}(x)$ только для иррациональных $x.$

Далее в этой части будут доказаны леммы об оценках сверху и снизу на величины
$$
\frac{g_{\varphi^{-1}}(x+\delta)-g_{\varphi^{-1}}(x)}{\delta}~~\text{и}~~\frac{g_{\tau}(x+\delta)-g_{\tau}(x)}{\delta}
$$
через $S^{\varphi}_t(x), S^{\tau}_t(x)$ и знаменатели подходящих дробей к иррациональному числу $x.$
Все леммы данного раздела будут доказаны только для случая $\delta>0$, поскольку случай $\delta<0$ аналогичен.

Прежде всего отметим следующее важное свойство.
\begin{lem}\cite[с. 23]{Khinchin}.
\label{medl}
\textit{Если $x$ - число, заключенное между двумя подходящими дробями $\frac{p_{l-1}}{q_{l-1}}$ и $\frac{p_l}{q_l},$ то $l+1$-е неполное частное $x$ - это максимальное $m$ такое, что $$\dfrac{p_{l+1}}{q_{l+1}}=\dfrac{p_{l-1}}{q_{l-1}}\oplus\underbrace{\dfrac{p_l}{q_l}\oplus\ldots\oplus\dfrac{p_l}{q_l}}_{m}$$ лежит по ту же самую сторону от~$x,$~что~и~$\frac{p_{l-1}}{q_{l-1}}.$}
\end{lem}
В условиях предыдущей леммы если $i<m,$ то дробь $$\dfrac{p_{l-1}}{q_{l-1}}\oplus\underbrace{\dfrac{p_l}{q_l}\oplus\ldots\oplus\dfrac{p_l}{q_l}}_{i}$$ называется \textit{промежуточной дробью} числа $x$.
\begin{lem}
\label{lemge}
Пусть $x=[a_1,\ldots,a_t,\ldots]$ - иррациональное число, тогда для любого достаточно малого по абсолютной величине $\delta$ существует $t=t(x,\delta)$ такое, что 
\begin{equation}
\begin{split}
\label{mainlemge}
\frac{g_{\varphi^{-1}}(x+\delta)-g_{\varphi^{-1}}(x)}{\delta}\geqslant\frac{q_{t}q_{t-1}}{\varphi^{S^{\varphi}_t(x)+7}},\\
\frac{g_{\tau}(x+\delta)-g_{\tau}(x)}{\delta}\geqslant\frac{q_{t}q_{t-1}}{\varphi^{S^{\tau}_t(x)+9}}.
\end{split}
\end{equation}
\end{lem}
\begin{proof}
Будем доказывать утверждение леммы для обеих функций параллельно. 

Пусть $\xi- $ промежуточная или подходящая дробь к числу $x$ с минимальным знаменателем, попавшая в интервал $(x, x+\delta).$ По лемме \ref{medl} она имеет вид:
$$
\xi=\dfrac{p_{l-1}}{q_{l-1}}\oplus\underbrace{\dfrac{p_l}{q_l}\oplus\ldots\oplus\dfrac{p_l}{q_l}}_{k+1}, ~~k\geqslant0.
$$
С другой стороны, $\xi$ представима в виде $\xi=\xi_0\oplus\xi_1,$ где среди дробей $\xi_0$ и $\xi_1$ одна подходящая к числу $x$, а вторая подходящая или промежуточная. Обозначим за $\xi_0$ меньшую из дробей, а за $\xi_1$ - большую. Поскольку знаменатели $\xi_0$ и $\xi_1$ меньше знаменателя $\xi$, то выполнены неравенства: 
$$\xi_0<x<\xi<x+\delta<\xi_1.$$
Нетудно видеть, что $\xi_0$ - подходящая дробь к $x$. Действительно, если  $\xi_0~-$ промежуточная дробь, то по лемме \ref{medl} медианта $\xi=\xi_0\oplus\xi_1$ должна лежать по ту же сторону от $x,$ что и $\xi_0,$ противоречие. Отметим также, что если $k=0,$ то обе дроби $\xi_0$ и $\xi_1$ являются подходящими к $x.$ В этом случае подходящая дробь $\xi_1$ имеет меньший порядок, поскольку медианта $\xi_0$ и $\xi_1$ лежит по ту же сторону от $x$, что и $\xi_1.$ Таким образом, 
$$
\xi_0=\frac{p_l}{q_l},~~
\xi_1=\dfrac{p_{l-1}}{q_{l-1}}\oplus\underbrace{\dfrac{p_l}{q_l}\oplus\ldots\oplus\dfrac{p_l}{q_l}}_{k}.
$$

Поскольку $\xi_0<x<\xi_1,$ подходящая дробь $\xi_0 $ имеет четный порядок, т.е. $\xi_0=[a_1,\ldots, a_{2t}].$
Рассмотрим 2 случая:\\ 
1)$~k>0.$\\
2)$~k=0.$\\
Разберем случай 1). Имеем:
\begin{equation}
\label{gto1}
\xi_1=[a_1,\ldots, a_{2t}, k],~{\xi=[a_1,\dots,a_{2t}, k+1],}~k+1\leqslant a_{2t+1}.
\end{equation}
Пусть теперь $z$ - такое минимальное натуральное число, что выполнено хотя бы одно из условий:
$$\xi_{-}=\xi_0\oplus\underbrace{\xi\oplus\ldots\oplus\xi}_{z}>x~\text{или}~ {\xi_{+}=\xi_1\oplus\underbrace{\xi\oplus\ldots\oplus\xi}_{z}<x+\delta}.$$
Введем также дроби
$$
\xi_{--}=\xi_0\oplus\underbrace{\xi\oplus\ldots\oplus\xi}_{z-1}<x~\text{и}~ {\xi_{++}=\xi_1\oplus\underbrace{\xi\oplus\ldots\oplus\xi}_{z-1}>x+\delta}.
$$
Таким образом,
\begin{equation}
\label{ofmon}
\xi_0\leqslant\xi_{--}<x<\xi<x+\delta<\xi_{++}\leqslant\xi_{1}.
\end{equation}
Из монотонности функций $g_{\varphi^{-1}}(x)$ и $g_{\tau}(x)$ следует, что
\begin{equation}
\label{gtmon}
g_{\varphi^{-1}}(x+\delta)-g_{\varphi^{-1}}(x)\geqslant \min{[g_{\varphi^{-1}}(\xi)-g_{\varphi^{-1}}(\xi_{-}), ~g_{\varphi^{-1}}(\xi_{+})-g_{\varphi^{-1}}(\xi)]},
\end{equation}
аналогичное утверждение верно для функции $g_{\tau}(x)$. Посмотрим, как разлагаются $\xi_{-}$ и $\xi_{+}$  в цепные дроби:
\begin{equation}
\begin{gathered}
\xi_{-}=\xi_0\oplus\underbrace{\xi\oplus\ldots\oplus\xi}_{z}=\frac{\langle a_2,\ldots,a_{2t}\rangle+z\langle a_2,\ldots,a_{2t},k+1\rangle}{\langle a_1,\ldots,a_{2t}\rangle+z\langle a_1,\ldots,a_{2t},k+1\rangle}=\\=\frac{\langle a_2,\ldots,a_{2t},k+1, z\rangle}{\langle a_1,\ldots,a_{2t},k+1,z\rangle}=[a_1,\ldots,a_{2t},k+1,z];\\
\xi_{+}=\xi_1\oplus\underbrace{\xi\oplus\ldots\oplus\xi}_{z}=\frac{\langle a_2,\ldots,a_{2t}, k\rangle+z\langle a_2,\ldots,a_{2t},k+1\rangle}{\langle a_1,\ldots,a_{2t}, k\rangle+z\langle a_1,\ldots,a_{2t},k+1\rangle}=\\=\frac{\langle a_2,\ldots,a_{2t}, k\rangle+z\langle a_2,\ldots,a_{2t},k,1\rangle}{\langle a_1,\ldots,a_{2t}, k\rangle+z\langle a_1,\ldots,a_{2t},k,1\rangle}=\frac{\langle a_2,\ldots,a_{2t},k,1, z\rangle}{\langle a_1,\ldots,a_{2t},k,1,z\rangle}=[a_1,\ldots,a_{2t},k,1,z].
\end{gathered}
\end{equation}
Аналогично, при $z>1$:
\begin{equation}
\label{gto2}
\xi_{--}=[0;a_1,\ldots,a_{2t},k+1,z-1],~ \xi_{++}={[0;a_1,\ldots,a_{2t},k,1,z-1]}.
\end{equation}
При $z=1$ легко видеть, что $\xi_{--}=\xi_0,~ \xi_{++}=\xi_1.$
Обозначим ${S^{\varphi}_t(x)-1=n_{\varphi}},\\ S^{\tau}_t(x)-1=n_{\tau}.$ Посчитаем теперь разности
$$g_{\varphi^{-1}}(\xi)-g_{\varphi^{-1}}(\xi_{-}), g_{\varphi^{-1}}(\xi_{+})-g_{\varphi^{-1}}(\xi)~\text{и}~g_{\tau}(\xi)-g_{\tau}(\xi_{-}), {g_{\tau}(\xi_{+})-g_{\tau}(\xi).}$$
Цепные дроби $\xi$ и $\xi_{-}$ отличаются только последним неполным частным, поэтому, ввиду (\ref{small}), (\ref{phiform}):
\begin{equation}
\begin{split}
\label{gt1}
g_{\varphi^{-1}}(\xi)-g_{\varphi^{-1}}(\xi_{-})=\frac1{\varphi^{n_{\varphi}+k+1+2z}},\\
g_{\tau}(\xi)-g_{\tau}(\xi_{-})=\frac1{\varphi^{n_{\varphi}+2k+2+z}}.
\end{split}
\end{equation}
Поскольку $\xi=[a_1,\dots,a_{2t}, k,1],$ то, аналогично,
\begin{equation}
\begin{split}
\label{gt2}
g_{\varphi^{-1}}(\xi_{+})-g_{\varphi^{-1}}(\xi)=\frac1{\varphi^{n_{\varphi}+k+2+z}},\\
g_{\tau}(\xi_{+})-g_{\tau}(\xi)=\frac1{\varphi^{n_{\varphi}+2k+2z+1}}.
\end{split}
\end{equation}
Из (\ref{gt1}) и (\ref{gt2}) несложно видеть, что
\begin{equation}
\begin{split}
g_{\varphi^{-1}}(\xi)-g_{\varphi^{-1}}(\xi_{-})\leqslant  g_{\varphi^{-1}}(\xi_{+})-g_{\varphi^{-1}}(\xi),\\
g_{\tau}(\xi)-g_{\tau}(\xi_{-})\geqslant g_{\tau}(\xi_{+})-g_{\tau}(\xi),
\end{split}
\end{equation}
а, следовательно, ввиду (\ref{gtmon})
\begin{equation}
\begin{split}
\label{lem1.1}
g_{\varphi^{-1}}(x+\delta)-g_{\varphi^{-1}}(x)\geqslant \frac1{\varphi^{n_{\varphi}+k+1+2z}},\\
g_{\tau}(x+\delta)-g_{\tau}(x)\geqslant \frac1{\varphi^{n_{\varphi}+2k+2z+1}}.
\end{split}
\end{equation}
Оценим теперь $\delta$. Рассмотрим два подслучая: \\
1.1)~$\xi$ -подходящая дробь, \\
1.2)~$\xi -$ промежуточная дробь.\\
Случай $1.1$ разбивается на 2 подслучая:\\
1.1.1)~$z=1$\\
1.1.2)~$z\ge2$\\
\\
1.1.1) Имеем ввиду леммы \ref{medl} $k+1=a_{2t+1},$ 
\begin{equation}
\label{o23}
\delta\leqslant\xi_1-\xi_0=\frac1{\langle a_1,\ldots,a_{2t}\rangle\langle a_1,\ldots,a_{2t},k\rangle}\leqslant \frac2{q_{2t}q_{2t+1}}.
\end{equation}
Таким образом, применяя оценки (\ref{o23}) и (\ref{lem1.1}) и учитывая, что $\varphi^2>2$, получаем:
\begin{equation}
\begin{gathered}
\label{c111}
\frac{g_{\varphi^{-1}}(x+\delta)-g_{\varphi^{-1}}(x)}{\delta}\geqslant \frac{q_{2t}q_{2t+1}}{\varphi^{n_{\varphi}+k+3}}=\frac{q_{2t}q_{2t+1}}{\varphi^{a_1+2a_2+\ldots+2a_{2t}+a_{2t+1}+2}},\\
\frac{g_{\tau}(x+\delta)-g_{\tau}(x)}{\delta}\geqslant \frac{q_{2t}q_{2t+1}}{\varphi^{n_{\tau}+2k+3}}=\frac{q_{2t}q_{2t+1}}{\varphi^{2a_1+a_2+\ldots+a_{2t}+2a_{2t+1}+1}}.
\end{gathered}
\end{equation}

$1.1.2)$ Аналогично, $k+1=a_{2t+1}.$ Поскольку $\xi_{--}<x$, то по лемме  \ref{medl} $z-1\leqslant a_{2t+2}.$ В этом случае получаем, используя (\ref{gto1}) и (\ref{gto2}):
\begin{multline}
\delta\le\xi_{++}-\xi_{--}=(\xi_{++}-\xi)+(\xi-\xi_{--})=\\=
\frac1{\langle a_2,\ldots,a_{2t},k,1, z-1\rangle\langle a_1,\ldots,a_{2t},k,1\rangle}+\frac1{\langle a_1,\ldots,a_{2t},k+1\rangle\langle a_1,\ldots,a_{2t},k+1,z-1\rangle}=\\=
\frac1{q_{2t+1}}(\frac1{\langle a_1,\ldots,a_{2t},k,1, z-1\rangle}+\frac1{\langle a_1,\ldots,a_{2t},k+1,z-1\rangle})\leqslant\\ \leqslant\frac2{q_{2t+1}(z-1)\langle a_1,\ldots,a_{2t},k+1\rangle}=\frac2{q^2_{2t+1}(z-1)}.
\end{multline}
То есть, 
\begin{equation}
\begin{split}
\label{in112}
\delta\le\frac2{q^2_{2t+1}(z-1)}.
\end{split}
\end{equation}
Пользуясь тем, что функция $\frac{x}{\varphi^{2x}}$ убывает на множестве натуральных чисел, из (\ref{lem1.1}), (\ref{in112}) получаем оценку:
\begin{equation}
\label{c112}
\frac{g_{\varphi^{-1}}(x+\delta)-g_{\varphi^{-1}}(x)}{\delta}\geqslant \frac{q^2_{2t+1}(z-1)}{\varphi^{n_{\varphi}+k+2(z-1)+5}}\geqslant \frac{q^2_{2t+1}(a_{2t+2}+1)}{\varphi^{n_{\varphi}+k+2(a_{2t+2}+1)+5}}\geqslant \frac{q_{2t+1}q_{2t+2}}{\varphi^{a_1+2a_2+\ldots+2a_{2t}+a_{2t+1}+2a_{2t+2}+6}}
\end{equation}
Для функции $g_{\tau}(x)$ разберем 2 подслучая:\\
\textbf{(i)}$~z-1\leqslant\frac{a_{2t+2}}2$. В этом случае, применяя оценки (\ref{lem1.1}) и (\ref{in112}), заключаем:
\begin{equation}
\label{c112m1}
\frac{g_\tau(x+\delta)-g_{\tau}(x)}{\delta}\geqslant \frac{q^2_{2t+1}(z-1)}{\varphi^{n_{\tau}+2k+2(z-1)+5}}\geqslant \frac{q^2_{2t+1}a_{2t+2}}{2\varphi^{n_{\tau}+2k+a_{2t+2}+5}}\geqslant \frac{q_{2t+1}q_{2t+2}}{\varphi^{2a_1+a_2+\ldots+a_{2t}+2a_{2t+1}+a_{2t+2}+7}}.
\end{equation}
Во втором неравенстве мы снова воспользовались монотонностью функции $\frac{x}{\varphi^{2x}}.$\\
\textbf{(ii)}~$z-1>\frac{a_{2t+2}}2.$\\
Напомним, что $\xi=\frac{p_{2t+1}}{q_{2t+1}}.$ Следовательно, по лемме \ref{medl} $x<\frac{p_{2t+2}}{q_{2t+2}}\oplus\xi<\xi.$ Отсюда по монотонности функции $g_{\tau}(x)$ имеем ввиду (\ref{small}):
\begin{equation}
g_{\tau}(x+\delta)-g_{\tau}(x)\geqslant g_{\tau}(\xi)-g_{\tau}(\xi\oplus\frac{p_{2t+2}}{q_{2t+2}})=
\frac1{\varphi^{n_{\tau}+2a_{2t+1}+a_{2t+2}+1}}.
\end{equation}
Теперь, применяя (\ref{in112}) и заменяя $z-1$ на $\frac{a_{2t+2}}2$, окончательно оцениваем:
 \begin{equation}
\label{c112m2}
\frac{g_\tau(x+\delta)-g_{\tau}(x)}{\delta}\geqslant \frac{q^2_{2t+1}(z-1)}{\varphi^{n_{\tau}+2a_{2t+1}+a_{2t+2}+1}}\geqslant \frac{q^2_{2t+1}a_{2t+2}}{2\varphi^{n_{\tau}+2a_{2t+1}+a_{2t+2}+1}}\geqslant \frac{q_{2t+1}q_{2t+2}}{\varphi^{2a_1+a_2+\ldots+a_{2t}+2a_{2t+1}+a_{2t+2}+5}}.
\end{equation}
1.2)~Поскольку $\xi$ - промежуточная дробь, то $z=1$ и 
$$
\delta\le\frac1{\langle a_1,\ldots,a_{2t}\rangle\langle a_1,\ldots,a_{2t}, k\rangle}\leqslant \frac1{kq^2_{2t}}.
$$
Поскольку $k<a_{2t+1}+1,$ имеем аналогично предыдущим случаям следующие оценки:
\begin{equation}
\begin{split}
\label{c12}
\frac{g_{\varphi^{-1}}(x+\delta)-g_{\varphi^{-1}}(x)}{\delta}\geqslant \frac{kq^2_{2t}}{\varphi^{n_{\varphi}+k+3}}\geqslant \frac{(a_{2t+1}+1)q^2_{2t}}{\varphi^{n_{\varphi}+a_{2t+1}+4}}\geqslant \frac{q_{2t+1}q_{2t}}{\varphi^{a_1+2a_2+\ldots+2a_{2t}+a_{2t+1}+4}},\\
\frac{g_{\tau}(x+\delta)-g_{\tau}(x)}{\delta}\geqslant \frac{kq^2_{2t}}{\varphi^{n_{\tau}+2k+3}}\geqslant \frac{(a_{2t+1}+1)q^2_{2t}}{\varphi^{n_{\tau}+2a_{2t+1}+3}}\geqslant \frac{q_{2t+1}q_{2t}}{\varphi^{2a_1+a_2+\ldots+a_{2t}+2a_{2t+1}+4}}.
\end{split}
\end{equation}
Теперь разберем 2), в этом случае соответствующие цепные дроби имеют следующий вид:
\begin{equation}
\begin{split}
\xi_1=[a_1,\ldots,a_{2t-1}],\\
\xi=[a_1,\ldots,a_{2t}+1],\\
\xi_{-}=[a_1,\ldots,a_{2t},1,z].\\
\end{split}
\end{equation}
Аналогично первому случаю оценим:
\begin{equation}
\begin{gathered}
g_{\varphi^{-1}}(x+\delta)-g_{\varphi^{-1}}(x)\geqslant \min{(g_{\varphi^{-1}}(\xi)-g_{\varphi^{-1}}(\xi_{-}), g_{\varphi^{-1}}(\xi_{+})-g_{\varphi^{-1}}(\xi))},\\
g_{\tau}(x+\delta)-g_{\tau}(x)\geqslant \min{(g_{\tau}(\xi)-g_{\tau}(\xi_{-}), g_{\tau}(\xi_{+})-g_{\tau}(\xi))},\\
g_{\varphi^{-1}}(\xi)-g_{\varphi^{-1}}(\xi_{-})=\frac1{\varphi^{n_{\varphi}+1+2z}}, \quad
g_{\tau}(\xi)-g_{\tau}(\xi_{-})=\frac1{\varphi^{n_{\tau}+z+1}},\\
g_{\varphi^{-1}}(\xi_{+})-g_{\varphi^{-1}}(\xi)=\frac1{\varphi^{n_{\varphi}+1+z}}, \quad
g_{\tau}(\xi_{+})-g_{\tau}(\xi)=\frac1{\varphi^{n_{\tau}+1+2z}}.
\end{gathered}
\end{equation}
А значит:
\begin{equation}
\begin{split}
\label{lem1.2}
g_{\varphi^{-1}}(x+\delta)-g_{\varphi^{-1}}(x)\geqslant \frac1{\varphi^{n_{\varphi}+1+2z}},\\
g_{\tau}(x+\delta)-g_{\tau}(x)\geqslant \frac1{\varphi^{n_{\tau}+1+2z}}.
\end{split}
\end{equation}
Поскольку оценки на величины $g_{\varphi^{-1}}(x+\delta)-g_{\varphi^{-1}}(x)$ и $g_{\tau}(x+\delta)-g_{\tau}(x)$ совпадают, то достаточно доказать утверждение только для функции $g_{\varphi^{-1}}(x)$

Оценим $\delta.$ Так же, как и в случае 1), рассмотрим два подслучая:\\
2.1)~$\xi$ -подходящая дробь;\\
2.2)~$\xi -$ промежуточная дробь.\\
Случай 2.1) дополнительно разбивается на два подслучая:\\
2.1.1)~$z=1;$\\
2.1.2)~$z\geqslant2.$\\
Разберем все случаи.\\
2.1.1)~Получаем $a_{2t+1}=1, \delta\leqslant\xi_1-\xi_0=\frac1{q_{2t-1}q_{2t}},$ следовательно
\begin{equation}
\label{c211}
\frac{g_{\varphi^{-1}}(x+\delta)-g_{\varphi^{-1}}(x)}{\delta}\geqslant \frac{q_{2t}q_{2t-1}}{\varphi^{n_{\varphi}+3}}=\frac{q_{2t}q_{2t-1}}{\varphi^{a_1+2a_2+\ldots+2a_{2t}+2}}
\end{equation}
2.1.2)~Аналогично, $a_{2t+1}=1,$ и ввиду леммы \ref{medl} $z-1\leqslant a_{2t+2}$. Кроме того, соответствующие цепные дроби в этом случае имеют вид:
\begin{equation}
\begin{split}
\xi_{--}=[a_1,\ldots,a_{2t},1,z-1],\\
\xi_{+}=[a_1,\ldots,a_{2t}+1,z],\\
\xi_{++}=[a_1,\ldots,a_{2t}+1,z-1].
\end{split}
\end{equation}
Следовательно $\delta$ можно оценить следующим образом:
\begin{multline}
\delta\le\xi_{++}-\xi_{--}=(\xi_{++}-\xi)+(\xi-\xi_{--})=\\=\frac1{\langle a_2,\ldots,a_{2t}+1, z-1\rangle\langle a_1,\ldots,a_{2t}+1\rangle}+\frac1{\langle a_1,\ldots,a_{2t}+1\rangle\langle a_1,\ldots,a_{2t},1,z-1\rangle}\leqslant \\ \leqslant \frac2{q_{2t+1}(z-1)q_{2t+1}}=\frac2{(z-1)q^2_{2t+1}}.
\end{multline}
Отсюда получаем:
\begin{equation}
\label{c212}
\frac{g_{\varphi^{-1}}(x+\delta)-g_{\varphi^{-1}}(x)}{\delta}\geqslant \frac{(z-1)q^2_{2t+1}}{\varphi^{n_{\varphi}+5+2(z-1)}}\geqslant \frac{(a_{2t+2}+1)q^2_{2t+1}}{\varphi^{n_{\varphi}+5+2(a_{2t+2}+1)}}\geqslant \frac{q_{2t+1}q_{2t+2}}{\varphi^{a_1+2a_2+\ldots+2a_{2t}+a_{2t+1}+2a_{2t+2}+7}}.
\end{equation}
2.2)~Аналогично случаю 1.2) из леммы \ref{medl} получчаем, что $z=1,$ тогда $\delta\leqslant\xi_1-\xi_0=\frac1{q_{2t-1}q_{2t}},$ следовательно
\begin{equation}
\label{c22}
\frac{g_{\varphi^{-1}}(x+\delta)-g_{\varphi^{-1}}(x)}{\delta}\geqslant \frac{q_{2t}q_{2t-1}}{\varphi^{n_{\varphi}+3}}=\frac{q_{2t}q_{2t-1}}{\varphi^{a_1+2a_2+\ldots+2a_{2t}+2}}.
\end{equation}
Объединяя все рассмотренные случаи, из формул (\ref{c111}), (\ref{c112}), (\ref{c112m1}), (\ref{c112m2}), (\ref{c12}), (\ref{c211}), (\ref{c212}) и (\ref{c22}) получаем утверждение леммы.
\end{proof}

\begin{lem}
\label{lemle}
Пусть $x=[a_1,\ldots,a_t,\ldots]$ - иррациональное число, тогда для любого достаточно малого по абсолютной величине $\delta$ существует $t=t(x,\delta)$ такое, что 
\begin{equation}
\begin{split}
\label{lem3}
\frac{g_{\varphi^{-1}}(x+\delta)-g_{\varphi^{-1}}(x)}{\delta}\leqslant \frac{q^2_{t}}{\varphi^{S^{\varphi}_t(x)-5}},\\
\frac{g_{\tau}(x+\delta)-g_{\tau}(x)}{\delta}\leqslant \frac{q^2_{t}}{\varphi^{S^{\tau}_t(x)-5}}.
\end{split}
\end{equation}
\end{lem}
\begin{proof}
Мы проведем доказательство только для функции $g_{\varphi^{-1}}(x),$ и при $\delta>0$, поскольку случаи, когда рассматриваемая функция - $g_{\tau}(x)$ или $\delta$ - отрицательно, совершенно аналогичны.\\ Таким же образом, как и в лемме \ref{lemge} определим числа $\xi,\xi_0, \xi_1, \xi_{++}, \xi_{--}, \xi_{+}, \xi_{-},n_{\varphi}$.\\ Ввиду (\ref{ofmon}) и монотонности функции $g_{\varphi}(x)$, выполнено неравенство
\begin{equation}
\begin{split}
g_{\varphi^{-1}}(x+\delta)-g_{\varphi^{-1}}(x)\leqslant  g_{\varphi^{-1}}(\xi_{++})-g_{\varphi^{-1}}(\xi_{--})
\end{split}
\end{equation}
 Рассмотрим 2 случая:\\
1)$~z=1.$ В этом случае $\xi_{--}=\xi_0, \xi_{++}=\xi_1.$ Следовательно, подставляя (\ref{gto1}) в (\ref{phiform}), получаем:
\begin{equation}
g_{\varphi^{-1}}(\xi_{++})-g_{\varphi^{-1}}(\xi_{--})=\frac1{\varphi^{n_{\varphi}+k}}.
\end{equation}
2)$~z\geqslant2.$ В этом случае из (\ref{gto1}), (\ref{gto2}) и формулы (\ref{phiform}) имеем
\begin{multline}
g_{\varphi^{-1}}(\xi_{++})-g_{\varphi^{-1}}(\xi_{--})= (g_{\varphi^{-1}}(\xi_{++})-g_{\varphi^{-1}}(\xi))+( g_{\varphi^{-1}}(\xi)-g_{\varphi^{-1}}(\xi_{--}))=\\=\frac1{\varphi^{n_{\varphi}+k+2z-1}}+\frac1{\varphi^{n_{\varphi}+k+1+z}}\leqslant \frac1{\varphi^{n_{\varphi}+k+z-1}}.
\end{multline}
Объединяя случаи, получаем 
\begin{equation}
g_{\varphi^{-1}}(\xi_{++})-g_{\varphi^{-1}}(\xi_{--})\leqslant \frac1{\varphi^{n_{\varphi}+k+z-1}}.
\end{equation}
Оценим теперь $\delta.$ Рассмотрим два случая:\\
1)~$\xi_{-}>x$\\
2)~$\xi_{-}<x, {\xi_{+}<x+\delta}$\\
1)Ввиду (\ref{ofmon}) $\delta>\xi-\xi_{-}.$ Рассмотрим еще $2$ подслучая:\\
1.1)~$z=1.$\\
1.2)~$z\geqslant2.$\\
Разберем все случаи:\\
1.1)~$\xi_{-}=[a_1,\ldots,a_{2t},k+2],~ k+2\le a_{2t+1}.$\\
Тогда
\begin{equation}
\xi-\xi_{-}=\frac1{\langle a_1,\ldots,a_{2t}, k+2\rangle\langle a_1,\ldots,a_{2t}, k+1\rangle}\geqslant \frac1{(k+3)^2q^2_{2t}}.
\end{equation}
Следовательно
\begin{equation}
\label{lc11}
\frac{g_{\varphi^{-1}}(x+\delta)-g_{\varphi^{-1}}(x)}{\delta}\leqslant \frac{q^2_{2t}(k+3)^2}{\varphi^{n_{\varphi}+k}}\leqslant \frac{q^2_{2t}}{\varphi^{n_{\varphi}-4}}=\frac{q^2_{2t}}{\varphi^{a_1+2a_2+\ldots+2a_{2t}-5}}.
\end{equation}
1.2)~Получаем по лемме \ref{medl} $\xi=[a_1,\ldots,a_{2t},a_{2t+1}], $ а $\xi_{-}=[a_1,\ldots,a_{2t},a_{2t+1},a_{2t+2}],$\\
${k+1=a_{2t+1}}.$\\
Следовательно
\begin{equation}
\delta>\xi-\xi_{-}=\frac1{\langle a_1,\ldots,a_{2t}, a_{2t+1}\rangle\langle a_1,\ldots,a_{2t},a_{2t+1},z\rangle}\geqslant \frac1{(z+1)q^2_{2t+1}}.
\end{equation}
А значит
\begin{equation}
\label{lc12}
\frac{g_{\varphi^{-1}}(x+\delta)-g_{\varphi^{-1}}(x)}{\delta}\leqslant \frac{(z+1)q^2_{2t+1}}{\varphi^{n_{\varphi}+a_{2t+1}+z-1}}\leqslant \frac{q^2_{2t+1}}{\varphi^{a_1+2a_2+\ldots+2a_{2t}+a_{2t+1}-1}}.
\end{equation}
$2)$ Аналогично прошлому случаю имеем\\
$\xi=[0;a_1,\ldots,a_{2t},a_{2t+1}],~ k+1=a_{2t+1}.$\\
Поскольку $\delta>\xi_{++}-\xi,$ получаем:
\begin{equation}
\delta>\xi_{++}-\xi=\frac1{\langle a_1,\ldots,a_{2t}, a_{2t+1}\rangle\langle a_1,\ldots,a_{2t},a_{2t+1}-1,1,z\rangle}\geqslant \frac1{(z+1)q^2_{2t+1}}.
\end{equation}
А значит, как и в случае 1.2
\begin{equation}
\label{lc2}
\frac{g_{\varphi^{-1}}(x+\delta)-g_{\varphi^{-1}}(x)}{\delta}\leqslant \frac{q^2_{2t+1}}{\varphi^{a_1+2a_2+\ldots+2a_{2t}+a_{2t+1}-3}}.
\end{equation}
Объединяя все рассмотренные случаи, из формул (\ref{lc11}), (\ref{lc12}) и (\ref{lc2}) получаем утверждение леммы.
\end{proof}

{\itshape Замечание.}
Леммы \ref{lemle} и \ref{lemge} являются непосредственным обобщением соответствующих лемм из \cite{MDK}.
\section{Сравнение континуантов}
 Существенным инструметном наших доказательств является сравнение континуантов.
Мы будем пользоваться различными методами из работы \cite{kkk} и их обобщениям.

Пусть $x$ имеет разложение в цепную дробь $[a_1, a_2, \ldots, a_t, \ldots].$ Обозначим:
\begin{equation}
\begin{split}
\varkappa_{\varphi}(x)=\lim\limits_{t\to\infty}  \frac{a_1+a_3+\ldots+a_{2t-1}}t+2\frac{a_2+a_4+\ldots+a_{2t}}t=\lim\limits_{t\to\infty}\frac{S^{\varphi}_t(x)}t\\
\varkappa_{\tau}(x)=\lim\limits_{t\to\infty}  2\frac{a_1+a_3+\ldots+a_{2t-1}}t+\frac{a_2+a_4+\ldots+a_{2t}}t=\lim\limits_{t\to\infty}\frac{S^{\tau}_t(x)}t.
\end{split}
\end{equation}
Из лемм \ref{lemge} и \ref{lemle} видно, что для того, чтобы оценивать производную $g_{\varphi^{-1}}(x)$ исходя из $\varkappa_{\varphi}(x)$, достаточно найти асимптотику максимума и минимума континуантов с заданным $S^{\varphi}_t(x)\sim \varkappa_{\varphi}(x)t$ при $t\to\infty,$ аналогично для $S^{\tau}_t(x).$ Обозначим через $M^{\varphi}(n, S_n)$ множество континуантов $\langle A\rangle$ с фиксированной длиной $n$ и фиксированной суммой $S^{\varphi}_n(A)=S_n$, введем также $\max(M^{\varphi}(n, S_n))$ и $\min(M^{\varphi}(n, S_n))$ - соответственно максимальное и минимальное значение континуантов из множества $M^{\varphi}(n, S_n)$. Для решения поставленной задачи достаточно уметь находить их с точностью до некоторой, не зависящей от $n$ и $A$ константы. То есть требуется найти такие функции $f^{\varphi}(n, S_n)$ и $g^{\varphi}(n, S_n),$ что\\
 $\max(M^{\varphi}(n, S_n)) \asymp f^{\varphi}(n, S_n),$\\
$\min(M^{\varphi}(n, S_n))\asymp g^{\varphi}(n, S_n).$\\
Без ограничения общности будем в дальнейшем считать, что $n$ четно. 

Аналогично можем определить множество  $M^{\tau}(n, S_n).$ Заметим, что между множествами  $M^{\varphi}(n, S_n)$ и  $M^{\tau}(n, S_n)$ существует биективное соответствие: если $\langle \overrightarrow{A}\rangle\in M^{\varphi}(n, S_n),$ то, поскольку все неполные частные $\langle\overrightarrow{A}\rangle$ при замене  $\langle\overrightarrow{A}\rangle\to\langle\overleftarrow{A}\rangle$ изменят четность индекса (т.к. $n$ четно),  $\langle \overleftarrow{A}\rangle$ лежит в множестве $M^{\tau}(n, S_n).$  А значит, так как $\langle\overrightarrow{A}\rangle=\langle\overleftarrow{A}\rangle,$  максимумы и минимумы по данным множествам совпадают. Поэтому достаточно исследовать только $\max^{\varphi}(n, S_n)$ и $\min^{\varphi}(n, S_n).$

Будем до конца данной части для простоты опускать верхний индекс $\varphi$ и писать просто $M(n, S_n), \max(M(n, S_n)), \min(M(n, S_n))$ и $S_n=S_n(A)$.

Для нахождения  $\max(n, S_n)$ и $\min(n, S_n)$ будем пользоваться следующим методом: пусть $\langle A\rangle=\langle a_1,\ldots, a_n\rangle$ - произвольный континуант из $M(n, S_n).$ Будем действовать на на него некоторыми преобразованиями, то есть изменять континуант так, чтобы длина и сумма $S_n$ сохранялась. Замену континуанта $\langle X\rangle$ на  $\langle Y\rangle$ обозначим $\langle X\rangle\to\langle Y\rangle.$ Будем пользоваться заменами следующего вида:\\
1)~Отражение - замена 
\begin{equation}
\langle\overrightarrow{P}, \overrightarrow{ Q}, \overrightarrow{R}\rangle\to\langle\overrightarrow{P}, \overleftarrow{Q}, \overrightarrow{R}\rangle
\end{equation}
где 
$Q=(a_i,a_{i+1}\ldots, a_{j-1},a_j)$, $i$ и $j$ имеют одинаковую четность.\\
Пример: замена $\langle1, 2, 3, 4, 5 \rangle\to\langle1,2,4,3,5\rangle,$ в данном случае $P=(1,2),\\ Q=(3,4), R=(5)$\\
2)~Единичная вариация (термин взят из \cite{MDK})~- замена
\begin{equation}
\begin{split}
\langle a_1,\ldots, a_{i-1},a_i,a_{i+1},\ldots, a_{j-1},a_j,a_{j+1},\ldots, a_n\rangle\to\\
\langle a_1,\ldots, a_{i-1},a_i-x,a_{i+1},\ldots,a_{j-1}, a_j+x,a_{j+1},\ldots, a_n\rangle
\end{split}
\end{equation}
где $x\in\mathbb{Z}, a_i-x>0, a_j+x>0,$  $i$ и $j$ имеют одинаковую четность.\\
Пример: замена $\langle1, 2, 3, 4, 5 \rangle\to\langle3,2,3,4,3\rangle,$\\
3) $(1,2)$-вариация - замена одного из двух видов. В первом случае одно неполное частное четного индекса уменьшается (увеличивается) на $x$, а другое неполное частное с нечетным индексом увеличивается (уменьшается) на $2x$.
\begin{equation}
\begin{split}
\langle a_1,\ldots, a_{i-1},a_i,a_{i+1},\ldots, a_{j-1},a_j,a_{j+1},\ldots, a_n\rangle\to\\
\langle a_1,\ldots, a_{i-1},a_i-x,a_{i+1},\ldots, a_{j-1},a_j+2x,a_{j+1},\ldots, a_n\rangle
\end{split}
\end{equation}
где $x\in\mathbb{Z}, a_i-x>0, a_j+2x>0,$  $i$ четно, а $j$ нечетно. Пример: 
$$
\langle1, 2, 3, 4, 5 \rangle\to\langle1,3,1,4,5\rangle 
$$
Во втором случае одно неполное частное четного индекса уменьшается (увеличивается) на $x$, а два других неполных частных с нечетным индексом соответственно увеличиваются (уменьшаются) на $x$.:
\begin{equation}
\begin{split}
\langle a_1,\ldots,a_{i-1},a_i,a_{i+1},\ldots,a_{j-1},a_j,a_{j+1},\ldots, a_k,\ldots, a_n\rangle\to\\
\langle a_1,\ldots,a_{i-1},a_i-x,a_{i+1},\ldots,a_{j-1},a_j+x,a_{j+1},\ldots, a_k+x,\ldots, a_n\rangle
\end{split}
\end{equation}
Где $x\in\mathbb{Z}, a_i-x>0, a_j+x>0, a_k+x>0,~i$ четно, а $j$ и $k$ нечетны.\\
Пример:
$$\langle1, 2, 3, 4, 5 \rangle\to\langle1,3,2,4,4\rangle$$

Очевидно, что все рассмотренные замены сохраняют длину $n$ и сумму $S_n$, то есть не выводят из множества $M(n,S_n)$. Нетрудно видеть, что действуя на произвольный континуант композицией указанных преобразований, можно получить любой континуант из множества  $M(n, S_n)$, в том числе минимальный и максимальный.

Задача поиска $\min(n, S_n)$ достаточно проста, ответ на это вопрос дает теорема \ref{mintheor}, доказанная ниже. Задача по нахождения максимума сложнее. Для иллюстрации метода его поиска сформулируем следующую лемму.
\begin{lem}
\label{minalgor}
Пусть существует такой континуант $\langle N_{max}\rangle,$ что для любого континуанта $\langle A\rangle\in M(n, S_n)$ найдется последовательность преобразований вида 1-3
\begin{equation}
\label{posl}
\langle A\rangle\to\langle A_1\rangle\to\langle A_2\rangle\to\ldots\to\langle A_t\rangle\to\ldots\to\langle N_{max}\rangle,
\end{equation}
где $\frac{\langle A_i\rangle}{\langle A_{i+1}\rangle}<2$, причем количество тех значений $i$, для которых ${\langle A_i\rangle>\langle A_{i+1}\rangle},$ ограничено некоторой, не зависящей от $A$ и $n,$ константой $m$. Тогда $\langle N_{max}\rangle\asymp max(n, S_n)$
\end{lem}
\begin{proof}
Пусть $\langle N_{max}\rangle$ - не максимальный континуант. Тогда применим к максимальному континуанту последовательность преобразований из условия, которая преобразует его в  $\langle N_{max}\rangle$. Очевидно, что он уменьшится не более, чем в $2^m$ раз, то есть  $\langle N_{max}\rangle$ отличается от максимума не более, чем в $2^m$ раз, где $m$ не зависит от $n$ и $A$. Что и требовалось доказать.
\end{proof}
Назовем последовательность преобразований из формулировки леммы \ref{minalgor} \textit{алгоритмом приведения к максимуму}. Аналогичная лемма, очевидно, верна для алгоритма приведения к минимуму.

Доказательство существования последовательности (\ref{posl}) мы будем строить до конца данной части. Для этого будем исследовать, какие преобразования типа 1-3 заведомо увеличивают (или уменьшают) континуант. Рассмотрим преобразование отражения, пусть $\langle\overrightarrow{ A}, \overrightarrow{ B}, \overrightarrow{C}\rangle$ заменяется на
$\langle\overrightarrow{ A}, \overleftarrow{ B}, \overrightarrow{C}\rangle$. В каких случаях можно однозначно утверждать, увеличивается ли при этом континуант?

В 1956 году Т.Моцкин и Е.Штраус доказали  следующую лемму:
\begin{lem} \cite{MS}.
\label{mslem}
Если для натуральных $a, b,c,d$ выполнено неравенство
$$
(b-f)(c-e)>0,
$$
то:
$$
\langle\overrightarrow{ A}, b, c, \overrightarrow{ B}, e, f, \overrightarrow{C}\rangle \geqslant  \langle\overrightarrow{ A}, b, e, \overleftarrow{B}, c, f, \overrightarrow{C}\rangle.
$$
При этом среди последовательностей неполных частных $A, B, C$ могут быть пустые.
\end{lem}
Неформально говоря, надо к большей цифре поставить большую цифру, чтобы увеличить континуант, например 
$$
43=\langle 1, 2, 3, 4 \rangle>\langle 1, 3, 2, 4 \rangle=40.
$$
В 2000 году И.Д.Кан получил следующее обобщение этого правила.
\begin{lem}\cite{Kan}.
\label{oms}
Неравенство $\langle\overrightarrow{ P}, \overrightarrow{ Q}, \overrightarrow{ R}\rangle\geqslant\langle\overrightarrow{ P}, \overleftarrow{ Q}, \overrightarrow{ R}\rangle$ выполнено тогда и только тогда, когда
\begin{equation}
([\overleftarrow{ P]}-[\overrightarrow{ R}])([\overleftarrow{ Q}]-[\overrightarrow{Q}])\geqslant0,
\end{equation}
причем неравенства могут обращаться в равенства только одновременно. Утверждение леммы остается верным, если среди наборов $A, B, C$ есть пустые (соответствующие цепные дроби тогда равны $0$)
\end{lem}
В частности, им была получена следующая формула:
\begin{equation}
\label{comparecont}
\frac{ \langle\overrightarrow{ P}, \overrightarrow{ Q}, \overrightarrow{ R}\rangle- \langle\overrightarrow{ P}, \overleftarrow{ Q}, \overrightarrow{ R}\rangle}{\langle\overrightarrow{ P}\rangle \langle\overrightarrow{ Q}\rangle \langle\overrightarrow{ R}\rangle}=([\overleftarrow{ P}]-[\overrightarrow{ R}])([\overrightarrow{ Q}]-[\overleftarrow{Q}])
\end{equation}
Отсюда выводится тривиальное, но полезное следствие
\begin{foll}\cite{MDK}.
В результате преобразования отражения континуант изменяется не более, чем в $2$ раза.
\end{foll}
\begin{proof}
Поскольку каждая из цепных дробей в числителе правой части равенства (\ref{comparecont}) не превосходит $1,$ имеем
\begin{equation}
|\langle\overrightarrow{P}, \overrightarrow{Q}, \overrightarrow{R}\rangle- \langle\overrightarrow{P}, \overleftarrow{Q}, \overrightarrow{R}\rangle|\leqslant \langle\overrightarrow{P}\rangle \langle\overrightarrow{Q}\rangle \langle\overrightarrow{R}\rangle
\end{equation}
А поскольку выполнены неравенства
\begin{equation}
\langle\overrightarrow{P}, \overrightarrow{Q}, \overrightarrow{R}\rangle>\langle\overrightarrow{P}\rangle \langle\overrightarrow{Q}\rangle \langle\overrightarrow{R}\rangle~ \text{и}~ \langle\overrightarrow{P}, \overleftarrow{Q}, \overrightarrow{R}\rangle>\langle\overrightarrow{P}\rangle \langle\overrightarrow{Q}\rangle \langle\overrightarrow{R}\rangle,
\end{equation}
то очевидно получаем требуемое.
\end{proof}

Оценим теперь, насколько меняется континуант при преобразованиях типа 2), то есть единичных вариациях. Докажем, что вместо максимума по  $M(n, S_n)$ можно искать максимум по меньшему множеству $M_4(n, S_n)\subset M(n,S_n)-$ по множеству континуантов, в котором все неполные частные одинаковой четности отличаются не более, чем на $1$ т.е. имееют вид $\{a, a+1\}$ и $\{b,b+1\}$ соответственно. Введем для краткости для произвольного континуанта $\langle A\rangle$ следующие обозначения:\\ $Odd(A)$ - множество $k\in \mathbb{N}$ таких, что $\exists j\in \mathbb{N}: j - \text{нечетно}, a_j=k.$\\
$Even(A)$ - множество $k\in \mathbb{N}$ таких, что $\exists j\in \mathbb{N}: j - \text{четно}, a_j=k.$\\
Введем также множество $N(A)=\{Odd(A), Even(A)\}$

Например, для $\langle A\rangle=\langle 1, 2, 1, 3, 5, 3 ,1, 4\rangle\\
Odd(A)=\{1,5\}, Even(A)=\{2,3,4\}, N(A)=(\{1,5\},\{2,3,4\}).$

Покажем, что для любого континуанта из $M(n, S_n)$ существует последовательность единичных вариаций, в которой все преобразования кроме, возможно, двух является увеличивающими, приводящая исходный континуант $\langle A\rangle$ в некоторый зависящий от него континуант $\langle A'\rangle,$ принадлежащий множеству $M_4(n, S_n),$ то есть $N(A')\subseteq(\{a,a+1\},\{b,b+1\})$ для некоторых натуральных $a$ и $b$. Это и будет означать, что максимум по множеству $M_4(n,S_n)$ не более, чем в константу раз отличается от максимума по $M(n, S_n).$\\
\begin{theor}[О единичной вариации]
\label{edvar}
$\max(M(n, S_n))\asymp\max(M_4(n, S_n)).$
\end{theor}
Будем доказывать теорему, действуя на исходный континуант $\langle A\rangle$ преобразованиями типа 2) так, чтобы он перешел в описанное множество, то есть рассмотрим последовательность континуантов
\begin{equation}
\label{poslcont}
\langle A\rangle\to\langle A_1\rangle\to\langle A_2\rangle\to\ldots\to\langle A_m\rangle
\end{equation}
такую, что $\langle A_m\rangle\in M_4(n,S_n)$
и для любого $i\leqslant m$ кроме, возможно, двух, $\langle A_i\rangle\geqslant\langle A_{i-1}\rangle,$ при этом $\langle A_i\rangle$ получается из $\langle A_{i-1}\rangle$ действием преобразования типа 2. Доказательство будет состоять из нескольких леммм.\\
\\
\textit{Уточнение параметров.}~
Пусть $\langle A\rangle=\langle a_1,a_2,\ldots,a_n\rangle\in M(n, S_n)$ не лежит в $M_4(n, S_n)$. Тогда в нем есть 2 элемента $a_i$ и $a_j$ с индексами одинаковой четности такие, что $|a_i-a_j|>1$. Запишем $a_i$ как $a+x$ и $a_j$ как $a-x$, сам континуант тогда примет вид 
$$\langle P, a+x, Q, a-x, R\rangle=f(x).$$
При этом если $a_i$ и $a_j$ одинаковой четности, то $a$ - целое и если они разной четности, то $a$ - полуцелое. Соответственно, $f(x)$ есть функция целого или полуцелого аргумента. 
Рассматривая континуанты при разных $x$, мы, очевидно, не выходим из $M(n, S_n)$. Найдем, при каких $x$ значение $f(x)$ максимально.

Следующая лемма представляет собой видоизменение соответствующей леммы из \cite{MDK}.
\begin{lem}
\label{parabola}
Максимум f(x) достигается в одной из следующих точек: ($-\frac{1}{2},  \frac{1}{2}$) (при $a$ полуцелом) или $(-1, 0, 1)$ (при $a$ целом).
\end{lem}
\begin{solve}
Докажем лемму для случая, когда $P$ и $R$ непусты. Применяя дважды (\ref{knu}), распишем конитнуант:
\begin{equation}
\begin{gathered}
\label{68a}
\langle A\rangle=\langle\underbrace{P},\underbrace{a+x, Q, a-x,} \underbrace{R}\rangle=\\=\langle P\rangle \langle a+x, Q, a-x\rangle \langle R\rangle+\langle P^-\rangle\langle Q, a-x\rangle\langle R\rangle+\\+\langle P \rangle\langle a+x, Q\rangle\langle R_-\rangle+\langle P\rangle\langle  Q^-_-\rangle\langle R\rangle.
\end{gathered}
\end{equation}
Будем использовать в сумме знак $O(1)$, означающий сумму не зависящих от $x$ членов, поскольку на максимум $f(x)$ они, очевидно, не влияют. В частности, в него можно сразу занести последний член правой части равенства (\ref{68a}). Продолжим равенство:
\begin{multline}
\langle A\rangle=\langle P\rangle \langle a+x, Q, a-x\rangle \langle R\rangle+\langle P^-\rangle\langle Q, a-x\rangle\langle R\rangle+\langle P \rangle\langle a+x, Q\rangle\langle R_-\rangle+O(1)=\\=\langle P\rangle \langle R\rangle (\langle a+x, Q \rangle(a-x)+\langle a+x, Q^-\rangle)+\langle P^-\rangle \langle R\rangle((a-x)\langle Q\rangle+\langle Q^-\rangle)+\\+\langle P\rangle \langle R_-\rangle((a+x)\langle Q\rangle +\langle Q_-\rangle)+O(1)=\langle P\rangle \langle R\rangle((a^2-x^2)\langle Q\rangle+(a-x)\langle Q_-\rangle+\\+(a+x)\langle Q^-\rangle)-x\langle P^-\rangle\langle Q\rangle\langle R\rangle +x\langle P\rangle\langle Q\rangle\langle R_-\rangle+O(1)=\\= -x^2\langle P\rangle\langle Q\rangle\langle R\rangle+x(\langle P\rangle\langle Q^-\rangle\langle R\rangle-\langle P\rangle\langle Q_-\rangle\langle R\rangle+\langle P\rangle\langle Q\rangle\langle R_-\rangle-\langle P^-\rangle\langle Q\rangle\langle R\rangle)+O(1).
\end{multline}
Получаем квадратный трехчлен, выразим координату его вершины~$x_m:$ 
\begin{equation}
\begin{gathered}
\label{xm}
x_m=\frac{\langle P\rangle\langle Q^-\rangle\langle R\rangle-\langle P\rangle\langle Q_-\rangle\langle R\rangle+\langle P\rangle\langle Q\rangle\langle R_-\rangle-\langle P^-\rangle\langle Q\rangle\langle R\rangle}{2\langle P\rangle\langle Q\rangle\langle R\rangle}=\\=\frac{[\overleftarrow{ Q}]-[\overrightarrow{ Q}]+[\overrightarrow{ R}]-[\overleftarrow{ P}]}{2}.
\end{gathered}
\end{equation}
Так как все цепные дроби в формуле (\ref{xm}) лежат на отрезке от $0$ до $1$, то, очевидно, ${-1<x_m<1}$, а значит, если $x_m>0$, то $f(1)>f(n+1), f(\frac12)>f(\frac12+n) ~\forall n \in \mathbb{N}$, аналогично для $x_m<0$. Случай, когда $P$ или $R$ пустые,~- аналогичен. Лемма доказана. 
\end{solve}
Таким образом, в случае, когда $a_i-a_j$ нечетно, замена\\
\begin{equation}
\label{pm12}
\langle P, a+x, Q, a-x, R\rangle\to\langle P, a\pm\frac12, Q, a\mp\frac12, R\rangle
\end{equation}
увеличивает континуант. При этом если $x_m\geqslant0$, в формуле (\ref{pm12}) сначала идет знак $+$, a затем $-$, а если $x_m\leqslant0$, то наоборот. Рассмотрим случай, когда $a_i-a_j$ четно. Из доказанной леммы следует, что если $|a_i-a_j|\geqslant4,$ то к этой паре неполных частных можно применить увеличивающую единичную вариацию.  Если же ${|a_i-a_j|=2,}$ то ситуация сложнее. Разбору этого случая и будет посвящено все дальнейшее доказательство теоремы. Прежде всего выведем из леммы \ref{parabola} важное следствие, которое мы будем неоднократно использовать в дальнейшем:
\begin{foll}
\label{mainfoll}
Если $|a_i-a_j|=2$ и $|x_m|\leqslant\frac12,$ то замена
$$\langle P, a\pm1, Q, a\mp1, R\rangle\to\langle P, a, Q, a, R\rangle$$
увеличивает континуант.
\end{foll}
\begin{proof}
Действительно, в этом случае, $f(0)\geqslant f(1)$ и $f(0)\geqslant f(-1),$ а следовательно максимум $f(x)$ по целым точкам достигается в точке $0,$ что и требовалось доказать. 
\end{proof}
Таким образом, применяя единичную вариацию, мы можем сделать так, чтобы все неполные частные континуанта $\langle A_k\rangle$ с индексами одинаковой четности отличались не более, чем на 2, где $A_k$ принадлежит последовательности континуантов (\ref{poslcont}), то есть
$$N(A_k)\subseteq(\{a-1, a, a+1\},\{b-1, b, b+1\}).$$
Если существуют неполные частные $a_i$ и $a_j$ с индексами одинаковой четности такие, что $a_i-a_j=2$, то рассмотрим замену
\begin{equation}
\label{var1}
\langle P', a_i, Q', a_j, R'\rangle\to\langle P', a_i-1, Q', a_j+1, R'\rangle.
\end{equation}
Не ограничивая общности, будем считать, что $i$ и $j$ четные. Тогда выполнено следующее:
\begin{lem}
\label{edvar1}
Если $Even(A_k)=\{a-1,a,a+1\},~a\ne2$ и $Odd(A_k)\nsubseteq\{1,2\}$, то существует единичная вариация, увеличивающая $\langle A_k\rangle$.
\end{lem}
\begin{solve}
Выберем  в $\langle A_k\rangle$ произвольные неполные частные\\ $a_i=a+1$ и $a_j=a-1,$ $i$ и $j$ четные. Рассмотрим замену, определенную формулой (\ref{var1}). Заметим, что если $1\notin Odd(A_k)$, то все цепные дроби из формулы (\ref{xm}) меньше $\frac12$, следовательно, $|x_m|<\frac12$, а значит по следствию \ref{mainfoll} замена (\ref{var1}) увеличивает континуант. 

Пусть теперь $\{1\}\in Odd(A_k),$ тогда $Odd(A_k)\subseteq\{1,2,3\}.$ Докажем, что если $\{3\}\in Odd(A_k)$, то увеличивающая единичная вариация существует. Действительно, поскольку по условию ${1\notin Even(A_k),}$ то применяя единичную вариацию к произвольным неполным частным, равным $1$ и $3$, мы можем сказать, что все цепные дроби в формуле (\ref{xm}) меньше $\frac12$, а значит по следствию \ref{mainfoll} замена
\begin{equation}
\label{13}
\langle P_1, 3, Q_1, 1, R_1\rangle\to\langle P_1, 2, Q_1, 2, R_1\rangle
\end{equation}
увеличивает континуант. Что и требовалось доказать.
\end{solve}
Докажем теперь, что в случае, когда  $Odd(A_k)\subseteq\{1,2\}$ также существует увеличивающая единичная вариация:
\begin{lem}
\label{edvar2}
Пусть $N(A_k)=(\{a-1, a, a+1\},\{1,2\}), a>2$, тогда единичная вариация, определенная формулой (\ref{var1}) увеличивает континуант.
\end{lem}
\begin{solve}
Воспользуемся Леммой \ref{parabola}. Максимальное значение цепных дробей из формулы \ref{xm} меньше либо равно $$[1,a+1,1]=\frac{a+2}{a+3}=1-\frac{1}{a+3},$$
а минимальное больше либо равно
$$[2,a-1]=\frac{a-1}{2a-1}=\frac12-\frac{1}{4a-2}.$$
Поэтому, подставляя данные оценки в формулу (\ref{xm}), получим:
\begin{multline}
|x_m|\leqslant  \frac{2(1-\frac{1}{a+3}-(\frac12-\frac{1}{4a-2}))}{2}=1-\frac{1}{a+3}-(\frac12-\frac{1}{4a-2})=\\
=\frac12-\frac{1}{a+3}+\frac{1}{4a-2}=\frac12-\frac{3a-5}{(a+3)(4a-2)}.
\end{multline}
Поскольку $\frac{3a-5}{(a+3)(4a-2)}>0$ при $a\ge2,$ имеем $|x_m|<\frac12$. Пользуясь следствием \ref{mainfoll}, получаем утверждение леммы. 
\end{solve}
Таким образом, осталось расмотреть случай, когда $N(A_k)\subseteq(\{1,2,3\},\{1,2,3\})$.
\begin{lem}
\label{edvar3}
Пусть $N(A_k)\subseteq(\{1,2,3\},\{1,2,3\})$, тогда для самой близкой в смысле разности индексов пары $(a_i, a_j)$ такой, что $a_i=3$ и $a_j=1,$ где $i$ и $j$ имеют одинакоую четность, замена (\ref{13}) увеличивает континуант.
\end{lem}
\begin{solve}
Без ограничения общности можем считать, что $i$ и $j$ четные. Рассмотрим замену (\ref{13}) и разность $[\overleftarrow{Q}]-[\overrightarrow{Q}]$ из формулы (\ref{xm}). Заметим, что все неполные частные $Q,$ имеющие в $\langle A_k\rangle$  четный индекс, равны $2$, т.к. иначе существовала бы более близкая пара с $1$ или $3$, а все неполные частные $Q,$ имеющие в $\langle A_k\rangle$ нечетный индекс, отличаются не более, чем на $1$ (т.е. равны $1$ и $2$ или $2$ и $3$). Таким образом
$${|[\overleftarrow{Q}]-[\overrightarrow{Q}]|\leqslant  [1,2,1,2 \ldots]-[2,2,2,2 \ldots]\leqslant [1,2,1]-[2,2]=\frac34-\frac25=\frac7{20}}.$$
Рассмотрим теперь внешнюю разность
$$|[\overrightarrow{R}]-[\overleftarrow{P}]|\leqslant  [1,3,1,3\ldots]-[3,1,3,1\ldots]\leqslant  [1,3,1]-[3,1]={\frac45-\frac14=\frac{11}{20}}.$$
Следовательно
$$|x_m|=\frac{|[\overleftarrow{Q}]-[\overrightarrow{Q}]|+|[\overrightarrow{R}]-[\overleftarrow{ P}]|}{2}\leqslant \frac{\frac{7}{20}+\frac{11}{20}}{2}=\frac{9}{20}<\frac12.$$
Отсюда по следствию \ref{mainfoll} и следует утверждение леммы.
\end{solve} 
Отметим, что в случае, когда мы применяем единичную вариацию к паре неполных частных, одно из которых является правым концом континуанта, соответствующая $x_m$  из формулы (\ref{xm}) равна $\frac{[\overleftarrow{ B}]-[\overrightarrow{ B}]-[\overleftarrow{ A}]}{2}$, что больше $-1$ и меньше $\frac{1}{2}$, аналогично для левого конца. В этих случаях единичная вариация может уменьшать континуант, но таких преобразований будет не более двух, и каждое уменьшит континуант не более, чем в $2$ раза.

Таким образом, из лемм \ref{parabola}-\ref{edvar3} следует, что если
$$N(A)\nsubseteq(\{a,a+1\},\{b,b+1\})$$
ни для каких натуральных $a$ и $b,$ то существует единичная вариация, увеличивающая $\langle A\rangle.$ Теорема \ref{edvar} доказана полностью.

Введем новое обозначение. Пусть дан произвольный континуант\\ $\langle C\rangle=\langle c_1,\ldots ,c_n\rangle,$ тогда обозначим через\\ $((c_{i_1}\to c'_{i_1}), (c_{i_1}\to c'_{i_1}),\ldots, (c_{i_k}\to c'_{i_k}))$\\ замену \\
$\langle c_1,c_2,\ldots, c_{i_1-1}, c_{i_1}, c_{i_1+1}\ldots,c_{i_2-1}, c_{i_2}, c_{i_2+1}\ldots,c_{i_k-1}, c_{i_k}, c_{i_k+1}\ldots,c_{n-1},c_n\rangle\to$\\
$\langle c_1,c_2,\ldots, c_{i_1-1}, c'_{i_1}, c_{i_1+1}\ldots,c_{i_2-1}, c'_{i_2}, c_{i_2+1}\ldots,c_{i_k-1}, c'_{i_k}, c_{i_k+1}\ldots,c_{n-1},c_n\rangle$,\\ то есть заменяем только элементы $c_{i_j},$ остальные неполные частные остаются теми же.\\
Докажем теперь теорему о минимуме.\\
\begin{theor}
\label{mintheor}
$\min(n,S_n)\asymp\langle\underbrace{1,\ldots,1,}_{n-1} s\rangle, $ где $s=S_n-\frac{3n-4}{2}.$
\end{theor}
\begin{proof}
Пусть $\langle A\rangle=\langle a_1,a_2,\ldots,a_n\rangle,$ выберем $a_{i}=\max(Even(A))$ - максимальное неполное частное четного индекса. Если существует четное $h$ такое, что $a_i=a_h,$ произведем замену
$$((a_i\to a_i+a_h-1), (a_h\to1)),$$
она увеличит континуант не более, чем в 2 раза. Тогда в новом континуанте $\langle A\rangle$ элемент $a'_i$ станет единственным максимальным неполным частным.

Из леммы \ref{parabola} следует, что поскольку график функции
$$f(x)=\langle A,a+x, B, a-x,C\rangle-$$ парабола с вершиной $x_m,~ |x_m|<1,$ то
$$
f(x+1)<f(x)~\forall x\geqslant\frac12~\text{и}~f(x-1)<f(x)~\forall x\leqslant-\frac12.
$$
Следовательно, любая замена
$$((a_i\to a_i+a_j-1), (a_j\to1)),~\text{где}~j-\text{четно,}$$
уменьшит континуант, поскольку при этом разность между неполными частными, для которых мы применяем единичную вариацию, увеличится.
Будем производить такие замены, пока все неполные частные четного индекса, кроме $a_i$, не станут равны $1$.

Произведем аналогичную процедуру для неполных частных нечетного индекса. Получим континуант содержащий не более 2 неполных частных, отличных от $1$. Он будет иметь вид
$$\langle1,\ldots, 1,\widetilde{a_i},1,\ldots,1,\widetilde{a_j},1,\ldots,1\rangle.$$
Если $a_j$ нечетно, произведем следующую замену
$$((\widetilde{a_i}\to \widetilde{a_i}+2\widetilde{a_j}),(\widetilde{a_j}\to 1))$$
Если же $a_j$ четно, то произведем другую замену
$$((\widetilde{a_i}\to \widetilde{a_i}+2\widetilde{a_j}-2),(\widetilde{a_j}\to 2))$$
Очевидно, что любая такая замена увеличит континуант не более, чем в $2$ раза.
Таким образом, полученный континуант имеет вид 
$$\langle\underbrace{1,\ldots,1}_{i-1},s,\underbrace{1,\ldots,1}_{n-i}\rangle~\text{или}~\langle\underbrace{1,\ldots,1}_{i-1},s,\underbrace{1,\ldots,1}_{j-i-1},2,\underbrace{1,\ldots,1}_{n-j}\rangle,$$
что не более, чем в константу раз отличается от $\langle\underbrace{1,\ldots,1,}_{n-1} s\rangle$, что и требовалось доказать.
\end{proof}
Введем новые обозначения:
\begin{equation}
\begin{split}
c^{(1)}_{a,a+1;b}=[b,a+1,b],~~c^{(2)}_{a,a+1;b}=[b+1,a],\\
c^{(1)}_{a;b,b+1}=[a,b+1,a],~~c^{(2)}_{a;b,b+1}=[a+1,b].
\end{split}
\end{equation}
\begin{lem}
Пусть $\langle A\rangle=\langle P,a,R\rangle=\langle P_1,b,R_1\rangle$-континуант, для которго $N(A)\subseteq(\{a,a+1\}, \{b,b+1\})$ и при этом $P, Q, P_1, Q_1$ состоят по крайней мере из $2$ неполных частных. Тогда выполнены следующие оценки:
\begin{multline}
\label{aa+1}\textbf{(i)}\quad
1+\frac1{a+2c^{(1)}_{a,a+1;b}}=\frac{a+1+2c^{(1)}_{a,a+1;b}}{a+2c^{(1)}_{a,a+1;b}}\leqslant \frac{\langle P,a+1,R\rangle}{\langle P,a,R\rangle}\leqslant\\ \leqslant\frac{a+1+2c^({2})_{a,a+1;b}}{a+2c^{(2)}_{a,a+1;b}}=1+\frac1{a+2c^{(2)}_{a,a+1;b}}
\end{multline}
\begin{multline}
\label{bb+1}
\textbf{(ii)}\quad1+\frac1{b+2c^{(1)}_{a;b,b+1}}=\frac{b+1+2c^{(1)}_{a;b,b+1}}{b+2c^{(1)}_{a;b,b+1}}\leqslant \frac{\langle P_1,b+1,R_1\rangle}{\langle P_1,b,R_1\rangle}\leqslant \\ \leqslant \frac{b+1+2c^{(2)}_{a;b,b+1}}{b+2c^{(2)}_{a;b,b+1}}=1+\frac1{b+2c^{(2)}_{a;b,b+1}}
\end{multline}
\end{lem}
\begin{proof}
Докажем первую оценку. Применяя равенства (\ref{A1}) и (\ref{knu}), получаем:
\begin{multline}
\frac{\langle P, a+1,R\rangle}{\langle P,a,R\rangle}=\frac{\langle P, a+1, R\rangle}{\langle P\rangle\langle R\rangle}\frac{\langle P\rangle\langle R\rangle}{\langle A,a,B\rangle}=\\\\=\frac{(a+1)\langle P\rangle\langle R\rangle+\langle P^-\rangle\langle R\rangle+\langle P\rangle\langle R_-\rangle}{\langle P\rangle\langle R\rangle}\frac{\langle P\rangle\langle R\rangle}{a\langle P\rangle\langle R\rangle+\langle P^-\rangle\langle R\rangle+\langle P\rangle\langle R_-\rangle}=\\ \\=\frac{a+1+[\overleftarrow{P}]+[\overrightarrow{R}]}{a+[\overleftarrow{P}]+[\overrightarrow{R}]}=1+\frac1{a+[\overleftarrow{P}]+[\overrightarrow{R}]}.
\end{multline}
Оценивая цепные дроби правой части последнего равенства снизу через $c^{(2)}_{a,a+1;b}$, а сверху через $c^{(1)}_{a,a+1;b}$, получаем оценку (\ref{aa+1}); оценка (\ref{bb+1}) доказывается аналогично. При этом мы пользуемся тем, что увеличение неполного частного нечетного индекса увеличивает цепную дробь, а увеличение неполного частного четного индекса, соответственно, уменьшает. Кроме того, любая подходящая к $x$ дробь четного порядка меньше $x,$  а нечетного порядка - больше $x$.

Соответственно, дробь $[b, a+1, b]$ является максимумом по множеству цепных дробей вида $[A],$ где $N(A)\subseteq(\{a,a+1\},\{b,b+1\})$ и длина $A$ больше $1$. По тем же причинам дробь $[b+1, a]$ является минимумом на описанном множестве цепных дробей. Второй случай абсолютно аналогичен. 
\end{proof}
Таким образом, мы получили верхние и нижние оценки изменения континуанта при заменах вида
$$(a\to a+1)~\text{и}~(b\to b+1)$$
Отдельно выделим формулу:
\begin{equation}
\label{kor}
\frac{\langle P, a+1, R\rangle}{\langle P,a,R\rangle}=\frac{a+1+[\overleftarrow{P}]+[\overrightarrow{R}]}{a+[\overleftarrow{P}]+[\overrightarrow{R}]}
\end{equation}
Отметим, что если $P$ или $R$ имеют длину меньше $2,$ то можно оценить цепные дроби сверху единицей, а снизу нулем, тогда формула  (\ref{kor}) превратится в
$$\frac43\leqslant\frac{\langle P, a+1, R\rangle}{\langle P,a,R\rangle}\leqslant\frac{a+3}{a+2}\leqslant\frac{a+1}{a}\leqslant2$$
Обозначим 
$$
c_l(a,a+1;b)=1+\frac1{a+2c^{(1)}_{a,a+1;b}}~\text{и}~c_r(a,a+1;b)=1+\frac1{a+2c^{(2)}_{a,a+1;b}}
$$
нижняя и верхняя оценки на величину $\frac{\langle P, a+1,R\rangle}{\langle P,a,R\rangle}$ из неравенства (\ref{aa+1}).\\
Аналогично определим
$$c_l(a;b,b+1)=1+\frac1{b+2c^{(1)}_{a;b,b+1}}~\text{и}~
c_r(a;b,b+1)=1+\frac1{b+2c^{(2)}_{a;b,b+1}}$$
нижнюю и верхнюю оценки на величину $\frac{\langle P', a+1,R'\rangle}{\langle P',a,R'\rangle}$ из неравенства (\ref{bb+1}).\\
Рассмотрим теперь замену
$$((a\to a+1), (b+1\to b), (b+1\to b))$$ 
в континуанте $\langle A\rangle,$ для которого $N(A)\subseteq(\{a,a+1\},\{(b,b+1\}),$ то есть замену \textit{любого} неполного частного с нечетным индексом, равного $a,$ на $a+1$ и замена \textit{любых} двух неполных частных с четным индексом, равных $b+1,$ на $b$. Нетрудно видеть, что рассмотренная замена является $(1, 2)-$вариацией. Выясним, пользуясь оценками предыдущей леммы, в каких случаях можно заведомо утверждать, что она увеличивает континуант. Для этого докажем следующее простое, но крайне полезное в дальнейшем утверждение.
\begin{lem}
\label{lemkor}
Пусть $\langle A\rangle$ - произвольный континуант, для которого выполнено $N(A)\subseteq(\{a,a+1\},\{b,b+1\}).$ Если при этом 
$$c_l(a,a+1,b)>c_r^2(a,b,b+1),$$ 
то замена
$${((a\to a+1), (b+1\to b), (b+1\to b))}$$
увеличивает континуант.\\ Если же
$$c_r(a,a+1,b)<c_l^2(a,b,b+1),$$ 
то замена
$$((a+1\to a), (b\to b+1), (b\to b+1))$$
увеличивает континуант.
\end{lem}
\begin{proof}
Для доказательства первого утверждения достаточно один раз применить неравенство (\ref{aa+1}) и дважды - неравенство (\ref{bb+1}). Второе утверждение доказывается аналогично.
\end{proof}
Назовем $(1,2)-$вариации, для которых выполняются условия леммы \ref{lemkor}, {\bfseries{абсолютно увеличивающими}}. Найдем конкретное выражение таких замен.
\begin{lem}
\label{a2a}
$(1,2)-$вариации
$$((a\to a+1),(2a+2\to 2a+1),(2a+2\to 2a+1))$$ при $N(A)\subseteq(\{a,a+1\},\{2a+1,2a+2\})$ и $a\geqslant1$\\
и
$$((a+1\to a),(2a\to 2a+1),(2a\to 2a+1))$$ при $N(A)\subseteq(\{a,a+1\},\{2a,2a+1\})$ и
$a\geqslant2$\\
являются абсолютно увеличивающими.
\end{lem}
\begin{proof}
Проверим выполнение условий предыдущей леммы.\\
Поскольку
$$c_l(a,a+1;2a+1)=\frac{4a^4+12a^3+21a^2+18a+7}{4a^4+8a^3+13a^2+9a+4}$$
$$c_r^2(a;2a+1,2a+2)=\frac{4(2a^3+5a^2+7a+3)^2}{(4a^3+8a^2+11a+4)^2},$$
то, сравнивая оценки, получаем:
\begin{multline}
\label{a1}
c_l(a,a+1;2a+1)-c_r^2(a;2a+1,2a+2)=\\=\frac{16a^8+96a^7+264a^6+432a^5+417a^4+198a^3-29a^2-92a-32}{(4a^3+8a^2+11a+4)^2(4a^4+8a^3+13a^2+9a+4)}
\end{multline}
что, очевидно, больше нуля при $a\geqslant1.$\\
Докажем аналогично вторую часть леммы: из
$$c_r(a,a+1;2a)=\frac{2a^3+3a^2+4a+1}{2a^3+a^2+3a},$$
$$c_l^2(a;2a,2a+1)=\frac{(4a^4+4a^3+9a^2+4a+2)^2}{4(2a^4+a^3+4a^2+a+1)^2},$$
получаем, что:
\begin{multline}
c_l^2(a;2a,2a+1)-c_r(a,a+1;2a)=\\=\frac{8a^9+12a^8+18a^7+13a^6-13a^5-24a^4-36a^3-28a^2-12a-4}{4a(2a^4+a^3+4a^2+a+1)^2(2a^2+a+3)},
\end{multline}
что больше нуля при $a\geqslant2.$ Лемма доказана.
\end{proof}
Доказанная лемма представляет собой "граничный" случай:\\ при $N(A)\subseteq(\{a,a+1\},\{2a,2a+1\})$ для увеличения континуанта необходимо увеличить неполные частные с четным индексом и уменьшить с нечетным, а при $N(A)\subseteq(\{a,a+1\},\{2a+1,2a+2\})$ - наоборот увеличить с нечетным и уменьшить с четным. Остальные случаи, как утверждает следующая лемма, проще:

\begin{lem}[Лемма о монотонности]
\label{monot}
Если замена $$((a\to a+1), (b+1\to b), (b+1\to b))~\text{при}~N(A)=(\{a,a+1\},\{b,b+1\}) -$$
абсолютно увеличивающая $(1,2)-$ вариация, то замены
$$((a-1\to a), (b+1\to b), (b+1\to b))~\text{при}~N(A)=(\{a-1,a\},\{b,b+1\}),$$
$$((a\to a+1), (b+2\to b+1), (b+2\to b+1))~\text{при}~{N(A)=(\{a,a+1\},\{b+1,b+2\})}$$
являются абсолютно увеличивающими.\\
Если же, напротив
$$((a+1\to a), (b\to b+1), (b\to b+1)) ~\text{при}~N(A)=(\{a,a+1\},\{b,b+1\})-$$ 
абсолютно увеличивающая замена, то $(1,2)-$вариации
$$((a+2\to a+1), (b\to b+1), (b\to b+1))~\text{при}~{N(A)=(\{a+1,a+2\},\{b,b+1\})},$$ 
$$((a+1\to a), (b-1\to b), (b-1\to b))~\text{при}~N(A)=(\{a,a+1\},\{b-1,b\})$$
также являются абсолютно увеличивающими заменами.
\end{lem}
\begin{proof}
Докажем первое утверждение. Поскольку первые неполные частные цепных дробей $c^{(1)}_{a,a+1;b}$ и $c^{(1)}_{a-1,a;b}$ совпадают, эти дроби отличаются не более, чем на $\frac12$. Следовательно, выполнена цепочка неравенств:
$$c_l(a,a+1;b)=1+\frac1{a+2c^{(1)}_{a,a+1;b}}<1+\frac1{a-1+2c^{(1)}_{a-1,a;b}}=c_l(a-1,a;b)$$ 
Сравним теперь $c_r^2(a;b,b+1)$ и $c_r^2(a-1;b,b+1)$. Они равны соответственно
$$(1+\frac1{b+2c^{(2)}_{a;b,b+1}})^2~ \text{и}~ (1+\frac1{b+2c^{(2)}_{a-1;b,b+1}})^2.$$
Заметим, что:
$$c^{(2)}_{a;b,b+1}=[a+1,b]<[a,b]=c^{(2)}_{a-1;b,b+1},$$
а следовательно 
$$c_r^2(a;b,b+1)>c_r^2(a-1;b,b+1)$$
Поскольку по условию $c_r^2(a,b,b+1)<c_l(a,a+1,b)$, получаем, что
$$c_r^2(a-1;b,b+1)<c_r^2(a;b,b+1)<c_l(a,a+1;b)<c_l(a-1,a;b).$$
Для завершения доказательства первого утверждения остается применить лемму \ref{lemkor}. Остальные утверждения доказываются аналогично.
\end{proof}
Во всех дальнейших леммах данной части мы будем пользоваться следующим, не ограничивающим общность, предположением:\\
если рассматривается замена
$$((a\to a+1), (b+1\to b), (b+1\to b))~\text{при}~N(A)\subseteq(\{a,a+1\},\{b,b+1\}),$$
то существует хотя бы 2 неполных частных нечетного индекса, равных $b+1;$ если же рассматривается замена
$$((a+1\to a), (b\to b+1), (b\to b+1)) ~\text{при}~N(A)\subseteq(\{a,a+1\},\{b,b+1\}),$$
то существует хотя бы 2 неполных частных нечетного индекса, равных $b.$
\begin{foll}
\label{mon}
Если ${N(A)=(\{a,a+1\},\{b,b+1\})}, a\geqslant2$, то для континуанта $\langle A\rangle$ существует абсолютно увеличивающая $(1,2)-$вариация.
\end{foll}
\begin{proof}
Действительно, если $b\leqslant2a,$ то по леммам \ref{a2a} и \ref{monot} замена
$$((a+1\to a),(b\to b+1),(b\to b+1))$$
является абсолютно увеличивающей. Если же  $b\geqslant2a+1,$ то аналогично замена
$$((a\to a+1),(b+1\to b),(b+1\to b))$$
является абсолютно увеличивающей, что и требовалось доказать.
\end{proof}
\begin{lem}
\label{to3}
Если $N(A)=(\{a\},\{b,b+1\})$ и при этом $b$ не равно $2a-1$ или $2a$ и $a>1,$ то для данного континуанта существует абсолютно увеличивающая $(1,2)-$вариация.
\end{lem}
\begin{proof}
Пусть $b<2a-1.$ Тогда рассмотрим замену
$$((a\to a-1),(b\to b+1),(b\to b+1))$$
Поскольку $N(A)\subset(\{a-1,a\},\{b,b+1\})$, по леммам \ref{a2a} и \ref{monot} она будет абсолютно увеличивающей, что и требовалось доказать. Аналогично рассматривается случай  $b>2a.$ 
\end{proof}
\begin{lem}
\label{to31}
Если $N(A)=(\{a,a+1\},\{b\})$ и при этом $b\ne2a+1$ и $a>1,$ то для данного континуанта существует абсолютно увеличивающая $(1,2)-$вариация.
\end{lem}
\begin{proof}
Пусть $b<2a-1.$ Тогда рассмотрим замену
$$((a+1\to a),(b\to b+1),(b\to b+1))$$
Поскольку $N(A)\subset(\{a,a+1\},\{b,b+1\})$, по леммам \ref{a2a} и \ref{monot} она будет абсолютно увеличивающей, что и требовалось доказать. Аналогично рассматривается случай  $b>2a-1.$ 
\end{proof}
\begin{lem}
\label{to2}
Если $N(A)=(\{a\},\{b\}), a\geqslant2$ и при этом $b$ не равно ${2a-1}, 2a$ или $2a+1,$ то для данного континуанта существует абсолютно увеличивающая $(1,2)-$вариация.
\end{lem}
\begin{proof}
Аналогично леммам \ref{to3} и \ref{to31}.
\end{proof}

Введем множество $M_3(n, S_n)\subset M_4(n, S_n)\subset M(n, S_n)$ - подмножество $M_4(n, S_n)$, состоящее из континуантов $\langle A\rangle$, для которых выполнено одно из следующих трех условий:\\
\begin{equation}
\begin{split}
\label{3t}
1)N(A)\subseteq(\{a\}, \{2a-1,2a\})\\
2)N(A)\subseteq(\{a\},\{2a,2a+1\})\\
3)N(A)\subseteq(\{a,a+1\},\{2a+1\})
\end{split}
\end{equation}
Во всех случаях считаем, что $a\geqslant2.$ Рассмотрим также $\max(M_3(n, S_n))$~- максимум по множеству $M_3(n, S_n).$\\
\begin{theor}[О сведении к трем неполным частным]
\label{thto3}
\quad\\
{Если $\frac{S_n}n\geqslant8,$ то  $\max(M(n, S_n))\asymp\max(M_3(n,S_n))$}
\end{theor}
\begin{proof}
Рассмотрим произвольный континуант $\langle A\rangle\in M(n, S_n).$ Докажем, что если он не лежит в $M_3(n, S_n),$ то существует последовательность увеличивающих преобразований, сохраняющих длину $n$ и $S_n$ и приводящих $\langle A\rangle$ в данное множество. Из теоремы \ref{edvar} можно считать, что данный континуант лежит в множестве $M_4(n, S_n)\smallsetminus M_3(n, S_n).$ Для этого покажем, что существует $(1,2)$-вариация, увеличивающая континуант $\langle A\rangle.$
Пусть $N(A)\subseteq(\{a,a+1\},\{b,b+1\}).$
Если $a>1$ и $\langle A\rangle\notin (M_3(n, S_n)),$ то существование увеличивающего преобразования прямо следует из лемм \ref{to3}, \ref{to31}, \ref{to2} и следствия \ref{mon}.

Рассмотрим случай $a=1.$\ Поскольку $\frac{S_n}n\geqslant8,~\text{то}~b\geqslant 3.$ Из первого утверждения леммы (\ref{a2a}) следует, что замена
$$((a\to a+1), (b+1\to b), (b+1\to b))$$
является абсолютно увеличивающей при $a=1, b=3,$ а следовательно, по лемме \ref{monot} она является также увеличивающей при $b>3.$

Таким образом, показано, что если континуант $\langle A\rangle\notin M_3(n, S_n),$ то для него существует абсолютно увеличивающая $(1,2)-$вариация. Что и требовалось доказать.
\end{proof}

Заметим также, что вид (одно из трех условий в (\ref{3t})), в который можно привести произвольный континуант в множестве  $M_3(n, S_n)$ однозначно определяется отношением $\frac{S_n}n$, поскольку для видов 1), 2) и 3) значение $\frac{S_n}n$ принадлежит, соответственно, отрезкам $[4a-1,4a], [4a,4a+1]$ или $[4a+1,4a+3].$
То есть отрезки пересекаются только по концам, соответствующим случаям, когда все неполные частные одинаковой четности совпадают. Таким образом, при фиксированном $a$ 
\begin{equation}
\label{snn}
\frac{S_n}n\in[4a-1,4a+3].
\end{equation}
\begin{foll}
\label{to47}
Если $\langle A\rangle\in M_3(n, S_n),~N(A)\subset(\{a,a+1\},\{b,b+1\})$ то $a\geqslant\frac{\frac{S_n}n-3}4,$ а $b\geqslant2a-1.$
\end{foll}
\begin{proof}
Первое неравенство очевидно следует из (\ref{snn}), а второе~- из (\ref{3t}).
\end{proof}
Осталось, таким образом, найти максимум по континуантам, в которых неполные частные могут принимать не более $3$ различных значений. Пусть для определенности $N(A)=(\{a,a+1\},\{b\}).$ Рассмотрим замены отражением $$\langle\overrightarrow{ P}, \overrightarrow{ Q}, \overrightarrow{ R}\rangle\to\langle\overrightarrow{ P}, \overleftarrow{ Q}, \overrightarrow{ R}\rangle,$$ где $Q=(a,b,\ldots,a+1)$ или $(a+1,b,\ldots,a),$ то есть $Q$ имеет разные начало и конец.

Пусть, например, $Q=(a,b,\ldots,a+1),$ тогда
$$[\overrightarrow{Q}]-[\overleftarrow{Q}]=[a,\ldots]-[a+1,\ldots]>0,$$
поэтому неравенство
$$\langle\overrightarrow{ P}, \overrightarrow{ Q}, \overrightarrow{ R}\rangle>\langle\overrightarrow{ P}, \overleftarrow{ Q}, \overrightarrow{ R}\rangle$$
выполняется тогда и только тогда, когда $[\overleftarrow{P}]>[\overrightarrow{R}].$\\В случае, если
$$[\overrightarrow{Q}]-[\overleftarrow{Q}]=[a+1,\ldots]-[a,\ldots]<0,$$
неравенство
$$\langle\overrightarrow{ P}, \overrightarrow{ Q}, \overrightarrow{ R}\rangle>\langle\overrightarrow{ P}, \overleftarrow{ Q}, \overrightarrow{ R}\rangle$$
выполнено тогда и только тогда, когда $[\overleftarrow{P}]<[\overrightarrow{R}].$
А поскольку все неполные частные с четными индексами совпадают, то цепные дроби $[\overleftarrow{ P}]$ и $[\overrightarrow{R}]$ имеют вид $[b, a_1,b,a_2\ldots],$ где $a_i\in\{a,a+1\}$. Обозначим их $[b,a^P_1, b, a^P_2, \ldots]$ и $[b,a^R_1, b, a^R_2,\ldots]$ соответственно. Сформулируем еще один критерий сравнения континаунтов.
\begin{lem}
\label{vrab}
Пусть $\langle A\rangle=\langle P, Q, R\rangle$ - произвольный континуант, причем $N(A)\subseteq(\{a, a+1\},\{b\})$, $Q=(q_1, \ldots, q_m)~\text{и}~q_1\ne q_m.$ Тогда:\\
$(\textbf{i})$ Если $q_1=a,  q_m=a+1,$ то замена
\begin{equation}
\label{mainotr}
\langle\overrightarrow{ P}, \overrightarrow{ Q}, \overrightarrow{ R}\rangle\to\langle\overrightarrow{ P}, \overleftarrow{ Q}, \overrightarrow{ R}\rangle
\end{equation}
увеличивает континуант тогда и только тогда, когда $\exists k: a^P_k<a^R_k$,\\причем $\forall i<K~a^P_i=a^R_i,$ то есть первое отчичающееся неполное частное цепных дробей $[\overleftarrow{ P}]$ и $[\overrightarrow{R}]$ больше в $[\overrightarrow{R}]$.\\
$(\textbf{ii})$ Если $q_1=a+1,  q_m=a,$ то замена, задаваемая формулой (\ref{mainotr}) увеличивает континуант тогда и только тогда, когда $\exists k: a^P_k>a^R_k$,\\ $\forall i<K~a^P_i=a^R_i,$ то есть первое отчичающееся неполное частное цепных дробей $[\overleftarrow{ P}]$ и $[\overrightarrow{R}]$ больше в $[\overrightarrow{P}]$.
\end{lem}
\begin{proof}
Заметим, что если цепные дроби $[\overleftarrow{ P}]$ и $[\overrightarrow{R}]$ отличаются неполным частным с четным индексом, то больше та дробь, у которой отличающееся неполное частное больше. Для завершения доказательства остается только воспользоваться леммой \ref{oms}.
\end{proof}
Заметим, что если одна из цепных дробей (более короткая) обрывается там, где заканчивается континуант, а все неполные частные более короткой цепной дроби совпадают с соответствующими неполными частными длинной, то следующее неполное частное короткой дроби можно считать равным $+\infty.$ В самом деле, если короткая цепная дробь состоит из нечетного количества неполных частных, то она больше длинной, а если из четного, то меньше, поскольку является подходящей дробью к более длинной.

Научимся теперь находить максимум по континуантам из множества $M_3(n, S_n)$. Без ограничения общности можем считать, что  произвольный континуант $\langle A\rangle\in M_3(n, S_n)$ состоит из блоков $(a,b)$ и $(a+1, b).$ Обозначим их $C^0_0$ и $C^0_1$ соответственно.
\begin{lem}
\label{simlem}
Если отношение количества блоков $C^0_0$ к количеству блоков $C^0_1$ равно ${m+\alpha}$, где 
${m\in\mathbb{N}},~{0<\alpha<1},$ то можно увеличивающими преобразованиями отражения добиться того, чтобы континуант состоял только из блоков
$$
C^1_0=(\underbrace{C^0_0,\ldots, C^0_0}_{m},C^0_1)~\text{и}~C^1_1=(\underbrace{C^0_0,\ldots, C^0_0}_{m+1},C^0_1).
$$
Если, напротив, отношение количества блоков $C^0_1$ к количеству блоков $C^0_0$ равно ${m+\alpha},$ где ${m\in\mathbb{N}},~{0<\alpha<1},$ то можно увеличивающими преобразованиями добиться того, чтобы континуант состоял только из блоков
$$
C^1_0=(\underbrace{C^0_1,\ldots, C^0_1}_{m},C^0_0)~\text{и}~C^1_1=(\underbrace{C^0_1,\ldots, C^0_1}_{m+1},C^0_0).
$$
Если, наконец, отношение количества блоков равно $m\in\mathbb{N},$ то максимум достигается на периодическом континуанте, состоящим из блоков
$$
(C^0_1\underbrace{C^0_0,\ldots, C^0_0}_{m})~\text{или}~(\underbrace{C^0_1,\ldots, C^0_1}_{m},C^0_0)
$$ в зависимости от того, каких блоков больше - $C^0_0$ или же $C^0_1.$ 
\end{lem}
\begin{proof}
Докажем первое утверждение леммы. Пусть в континуанте встречается блок  $C_k=(C^0_1,\underbrace{C^0_0,\ldots, C^0_0}_{m-k},C^0_1),$ где $k>0.$ Поскольку отношение числа блоков $C^0_0$ и $C^0_1$ больше $m,$ то существует блок $C_t=(\underbrace{C^0_0,\ldots, C^0_0}_{m+t}),$ где ${ t>0}.$ Пусть блок $C_k$ встречается в континуанте раньше $C_t$. Тогда посмотрим на предпоследнее неполное частное блока $C_k,$ равное $a_i=a+1$ и первое неполное частное блока $C_t,$ равное $a_j=a.$ Континуант в этом случае имеет вид
$$
\langle\overbrace{\ldots,C_1^0,\underbrace{C^0_0,\ldots, C^0_0,}_{m-k}}^{\overrightarrow{P}}\underbrace{\overbrace{a+1}^{a_i},b,\ldots, \overbrace{a}^{a_j}}_{\overrightarrow{Q}},\overbrace{b,\underbrace{C^0_0,\ldots, C^0_0}_{m+t-1}\ldots}^{\overrightarrow{R}}\rangle
$$
Тогда по лемме \ref{vrab} отражение набора $Q$ увеличивает континуант, поскольку первый отличающийся блок за $a_i$ равен $C^0_1=(a+1, b)$ , а за  $a_ j$ - $C^0_0=(a ,b),$ что и требовалось доказать. Если же напротив $C_t$ идет раньше $C_k,$ то отражение $Q$
$$
\langle\overbrace{\ldots,\underbrace{C^0_0,\ldots, C^0_0}_{m+t-1}}^{\overrightarrow{P}}\underbrace{a, b,\ldots,a+1,}_{\overrightarrow{Q}}\overbrace{b,\underbrace{C^0_0,\ldots, C^0_0,}_{m-k}C_1^0,\ldots}^{\overrightarrow{R}}\rangle
$$
аналогично увеличивает континуант. Таким же образом поступаем, если существует блок $C_t=(\underbrace{C^0_0,\ldots, C^0_0}_{m+t+1}),~{ t>0}.$ В этом случае найдется блок ${C_k=(C^0_1,\underbrace{C^0_0,\ldots, C^0_0}_{m-k},C^0_1})$, $k\geqslant0,$ и аналогичная замена увеличит континуант.

Докажем теперь второе утверждение. Пусть  существет блок\\
$C_k=(C^0_0,\underbrace{C^0_1,\ldots, C^0_1}_{m-k},C^0_0)$, где $k>0.$ По тем же соображением найдется блок $C_t=(\underbrace{C^0_1,\ldots, C^0_1}_{m+t}), { t>0}.$ Аналогично предположим, что блок $C_k$ идет раньше $C_t.$ Тогда рассмотрим замену отражением $Q:$
$$
\langle\overbrace{\ldots,C_0^0,\underbrace{C^0_1,\ldots, C^0_1,}_{m-k}}^{\overrightarrow{P}}\underbrace{a,b,\ldots, a+1}_{\overrightarrow{Q}},\overbrace{b, \underbrace{C^0_1,\ldots, C^0_1,}_{m+t-1}\ldots}^{\overrightarrow{R}}\rangle
$$
При данной замене первое различие цепных дробей $[\overleftarrow{P}]$ и $[\overrightarrow{Q}]$ будет в $m-k+1-$ом нечетном неполном частном. Поскольку у $[\overleftarrow{P}]$ оно меньше, то отражение $Q$, которое начинается с $a$ и заканчивается на $a+1$ увеличивает континуант. Вторая часть доказывается аналогично.

Докажем третье утверждение леммы. Пусть доля $C_0^0$ больше. Если существует блок $C_k$, в котором идут менее $m-1$ раз подряд идет $C_0^0$, то существует  блок, в котором идут менее $m-1$ раз подряд идет $C_1^0$, а значит, аналогично первым двум частям первым двум частям, если $C_k$ не является началом или концом континуанта, существет увеличивающая замена. Рассмотрим теперь концы континуанта. Из леммы \ref{mslem} сразу следует, что континуант должен начинаться с $C_1^0$ и заканчиваться на $C^0_0.$ Пусть есть нарушение блоковой структуры в начале, то есть континуант имеет вид
$$\langle C^0_1,\underbrace{C^0_0,\ldots,C^0_0}_{m-k},C^0_1,\ldots, \underbrace{C^0_0,\ldots, C^0_0}_{m+t},\ldots\rangle,~\text{где}~k>0,t>0$$

В этом случае, очевидно, работает тот же самые прием, что и в первой части. Если же нарушение в конце, то есть
$$\langle\ldots, \underbrace{C^0_0,\ldots, C^0_0}_{m+t},\ldots,C^0_1,\underbrace{C^0_0,\ldots,C^0_0}_{m-k}\rangle,~ k>0,t>0,$$ то отражение $Q$
$$\langle\overbrace{\ldots, \underbrace{C^0_0,\ldots, C^0_0}_{m+t-1}}^{\overrightarrow{P}},\underbrace{a,b,\ldots,a+1}_{\overrightarrow{Q}},\overbrace{b,\underbrace{C^0_0,\ldots,C^0_0}_{m-k}}^{\overrightarrow{R}}\rangle,~\text{где}~ k>0,~t>0,$$ увеличивает континуант, поскольку первое отличающееся неполное частное $[\overleftarrow{P}]$ и $[\overrightarrow{R}]$ у  $[\overrightarrow{R}]$ равно $+\infty$ и имеет четный индекс. Что и требовалось доказать.
\end{proof}
{\itshape Замечание.}~
Нетрудно доказать, что периодический континуант, составленный из блоков
$(C^0_1\underbrace{C^0_0,\ldots, C^0_0}_{m}),$ не более, чем в $2$ раза отличается от континуанта той же длины, составленного из блоков  $(\underbrace{C^0_0,\ldots, C^0_0}_{m},C^0_1).$

Приведем теперь индуктивное обобщение доказанной леммы. Для этого дадим индуктивное определение блоковой структуры i-го уровня.\\
1) Континуант имеет блоковую структуру $0$-го уровня, если его последовательность неполных частных можно представить в виде последовательности блоков $C^0_0$ и $ C^0_1.$ Как было сказано выше, мы без ограничения общности считаем, что любой континуант из $M_3(n, s_n)$ имеет блоковую структуру $0-$го уровня.\\
2) Если континуант имеет блоковую структуру $k-$го уровня, то есть представим в виде
$$
\langle C^k_{i_1},C^k_{i_2},\ldots,C^k_{i_m}\rangle, i_l\in\{0,1\}
$$
и при этом его также можно представить в виде
\begin{equation}
\label{jl}
\langle C^{k+1}_{j_1},C^{k+1}_{j_2},\ldots,C^{k+1}_{j_n}\rangle, j_l\in\{0,1\},
\end{equation}
где
$$
C^{k+1}_0=(\underbrace{C^k_{b},\ldots, C^k_{b}}_{n},C^k_{s}),\quad C^{k+1}_1=(\underbrace{C^k_{b},\ldots, C^k_{b}}_{n+1},C^k_{s}),~~b+s=1,
$$
то такое представление назовем блоковой структурой  $k+1-$го уровня. Блок $C^k_b$ мы назовем доминирующим блоком, а парный ему блок $C^l_s$~- доминируемым. Назовем блоковую структуру $k+1-$го уровня вырожденной, если все $j_l$ в (\ref{jl}) одновременно равны между собой, и невырожденной в противном случае. Лемму \ref{simlem} можно с помощью новых определений сформулировать в следующем, более кратком виде:\\
\textit{\textbf{(i)}~Если континуант имеет невырожденную блоковую структуру $0-$го уровня, то его можно при помощи увеличивающих преобразований перевести в континуант, имеющий блоковую структуру $1$-го уровня.\\
\textbf{(ii)}~Если в континуанте блоковая структура $1$-го уровня вырождена, то данный континуант является максимумом по множеству $M_3(n, S_n)$ с точностью до некоторой, не зависящей от $n$ константы.}

Итак, пусть континуант состоит из блоков $k-$го уровня $C^k_0$ и $C^k_1$. Тогда проведем индуктивный переход к блокам $k+1-$го уровня:

\begin{theor}[Рекурсивный алгоритм поиска максимума]
\label{maxs}
\quad\\
\textbf{(i)}~Если континуант имеет невырожденную блоковую структуру $k-$го уровня, то его можно при помощи увеличивающих преобразований перевести в континуант, имеющий блоковую структуру $k+1$-го уровня.\\ 
\textbf{(ii)}Если в континуанте блоковая структура $k+1$-го уровня вырождена, то данный континуант отличается от максимума по множеству $M_3(n, S_n)$ не более чем в $8$ раз.
\end{theor}
Доказательству утверждения \textbf{(i)} предпошлем ряд вспомогательных лемм и следствий.
Прежде всего, изучим более подробно структуру блоков $k+1-$го уровня.
Обозначим $$C^k_{tail}=(b,C^0_s,C^1_s,\ldots, C^{k-1}_s).$$ Легко видеть, что ${C^k_{tail}=(C^{k-1}_{tail},C^{k-1}_s.})$ Назовем $C^k_{tail}$ \textit{хвостом $k$-го уровня}.
Введем еще одно обозначение: пусть последовательность неполных частных $X$ представима в виде $X=(A, B).$ Тогда в качестве $X\smallsetminus B$ мы будем обозначать $A$. Если же данное представление $X$ не имеет места, то обозначение $X\smallsetminus B$ некорректно.
\begin{lem}[Лемма о существовании хвоста]
Любой блок $k-$го уровня $C^k_i$ представим в виде $(C^k_i\smallsetminus C^k_{tail}, C^k_{tail}),$ $i\in\{0,1\}.$
\end{lem}
\begin{proof}
Утверждение несложно доказывается по индукции. Для $k=0$ оно очевидно, $C^k_{tail}$ в этом случае равно $b$. Пусть утверждение верно для всех $k<l+1$, тогда, пользуясь предположением индукции, получаем:
\begin{equation}
C^{l+1}_0=(\underbrace{C^l_b,\ldots, C^l_b}_{n},C^l_{s})=(\underbrace{C^l_b,\ldots, C^l_b}_{n-1},C^l_b\smallsetminus C^l_{tail}, \underbrace{C^l_{tail},C^l_{s}}_{C^{l+1}_{tail}}).
\end{equation}
Аналогично для $C^{l+1}_1.$
\end{proof}
\begin{lem}[Основная лемма о хвосте]
\label{mtail}
$\overleftarrow{C^k_i}=({C^k_{tail},C^k_i\smallsetminus C^k_{tail}}).$
\end{lem}
\begin{proof}
Докажем по индукции. Для $k=0$ очевидно, что $$\overleftarrow{C^0_0}=(b,C^0_0\smallsetminus b),~~ \overleftarrow{C^0_1}=(b,C^0_1\smallsetminus b).$$
Пусть утверждение верно для всех $k\leqslant l$; рассмотрим $k=l+1:$
$$
C^{l+1}_0=(\underbrace{C^l_{b},\ldots, C^l_{b}}_{n},C^l_{s}), \quad \overleftarrow{C^{l+1}_0}=(\overleftarrow{C^l_s},\underbrace{\overleftarrow{C^l_{b}},\ldots, \overleftarrow{C^l_{b}}}_{n}).
$$
Пользуясь предположением индукции, получаем:
\begin{multline}
\overleftarrow{C^{l+1}_0}=({\overbrace{C^l_{tail},C^l_s\smallsetminus C^l_{tail}}^{\overleftarrow{C^l_s}}},\underbrace{\overbrace{C^l_{tail},{C^l_b\smallsetminus C^l_{tail}}}^{\overleftarrow{C^l_b}},\ldots,\overbrace{C^l_{tail},{C^l_b\smallsetminus C^l_{tail}}}^{\overleftarrow{C^l_b}}}_{n})=\\=(C^l_{tail},C^l_s,\underbrace{C^l_b,\ldots, C^l_b,C^l_b\smallsetminus C^l_{tail}}_{n})=(C^{l+1}_{tail},\underbrace{C^l_b,\ldots, C^l_b,C^l_b\smallsetminus C^l_{tail}}_{n})=\\=(C^{l+1}_{tail},{C^{l+1}_0\smallsetminus C^{l+1}_{tail}}).
\end{multline}
Абсолютно аналогично утверждение доказывется для $C^{l+1}_1.$ Лемма доказана.
\end{proof}
Выведем три простых следствия доказанной леммы.
\begin{foll}
\label{folltail}
$\overleftarrow{C^k_i\smallsetminus C^k_{tail}}=C^k_i\smallsetminus C^k_{tail}, \quad \overleftarrow{C^k_{tail}}=C^k_{tail}.$
\end{foll}
\begin{proof}
По предыдущей лемме имеем:
$$
\overleftarrow{C^k_i}=(C^k_{tail}, C^k_i\smallsetminus C^k_{tail}).
$$
С другой стороны,
$$
\overleftarrow{C^k_i}=\overleftarrow{(C^k_i\smallsetminus C^k_{tail},C^k_{tail})}=(\overleftarrow{C^k_{tail}}, \overleftarrow{C^k_i\smallsetminus C^k_{tail}}).
$$
Что и требовалось доказать.
\end{proof}
\begin{lem}[Обобщенная лемма о хвосте]
\label{bigtail}
$$(\overleftarrow{C^k_{i_1},C^k_{i_2},\ldots, C^k_{i_n}})=(C^k_{tail},C^k_{i_n},\ldots,C^k_{i_2},C^k_{i_1}\smallsetminus C^k_{tail}),~~i_j\in\{0,1\}.$$
\end{lem}
\begin{proof}
Очевидно, что
$$
(\overleftarrow{C^k_{i_1},C^k_{i_2},\ldots, C^k_{i_n}})=(\overleftarrow{C^k_{i_n}},\ldots, \overleftarrow{C^k_{i_2}},\overleftarrow{C^k_{i_1}}).
$$
Применяя лемму \ref{mtail} к каждому из блоков $C^k_{i_j}$, получаем
\begin{multline}
(\overleftarrow{C^k_{i_n}},\ldots, \overleftarrow{C^k_{i_2}},\overleftarrow{C^k_{i_1}})=(C^k_{tail},C^k_{i_n}\smallsetminus C^k_{tail},\ldots, C^k_{tail},C^k_{i_2}\smallsetminus C^k_{tail}, C^k_{tail},C^k_{i_1}\smallsetminus C^k_{tail})=\\=(C^k_{tail},C^k_{i_n},\ldots,C^k_{i_2},C^k_{i_1}\smallsetminus C^k_{tail}).
\end{multline}
\end{proof}
\begin{foll}
\label{bigtail2}
$$(\overleftarrow{C^k_{i_1},C^k_{i_2},\ldots, C^k_{i_n}\smallsetminus C^k_{tail}})=(C^k_{i_n},\ldots,C^k_{i_2},C^k_{i_1}\smallsetminus C^k_{tail}),~~i_j\in\{0,1\}.$$
\end{foll}
\begin{proof}
Утверждение очевидно следует из леммы \ref{bigtail} и следствия \ref{folltail}.
\end{proof}
\begin{lem}
\label{simq}
Если в континуанте встречаются последовательности неполных частных $$(C_1^k,\underbrace{C^k_0,\ldots, C^k_0}_{m-k},C_1^k)$$ и $$(\underbrace{C^k_0,\ldots, C^k_0}_{m+t})$$ для каких-то натуральных $k$ и $t,$ то для данного континуанта существует абсолютно увеличивающая замена отражением. 
\end{lem}
\begin{proof}
Рассмотрим следующее разбиение континуанта:
\begin{equation}
\label{mainotr}
\langle\overbrace{\ldots,C_1^k,\underbrace{C^k_0,\ldots, C^k_0,}_{m-k}}^{\overrightarrow{P}}\overbrace{C^k_1,\ldots, C^k_0\smallsetminus C^k_{tail}}^{\overrightarrow{Q}},\overbrace{ C^k_{tail},\underbrace{C^k_0,\ldots, C^k_0}_{m+t-1},\ldots}^{\overrightarrow{R}}\rangle.
\end{equation}
Заметим, что первое отличающееся неполное частное в цепных дробях $[\overleftarrow{P}]$ и $[\overrightarrow{R}]$ есть первое отличающееся неполное частное дробей $[C^k_1]$ и $[C^k_0]$. Действительно, 
$$[\overrightarrow{R}]=[ C^k_{tail},\underbrace{C^k_0,\ldots, C^k_0}_{m+t-1},\ldots].$$
С другой стороны, по лемме \ref{bigtail}  
$$[\overleftarrow{P}]=[C^k_{tail},\underbrace{C^k_0,\ldots, C^k_0}_{m-k},C^k_1,\ldots],$$ 
что и требовалось доказать.

Рассмотрим $Q=(q_1,\ldots,q_l),$ найдем такое минимальное $i$, что\\ $q_i\ne q_{l+1-i}$. Для этого сравним $\overrightarrow{Q}$ и $\overleftarrow{Q}$. По следствию \ref{bigtail2}
$$
\overleftarrow{Q}=(\overleftarrow{C^k_0\smallsetminus C^k_{tail}}\overleftarrow{C^k_0}\ldots),
$$
что по лемме \ref{bigtail} и следствию \ref{folltail} равно
$$
(C^k_0\smallsetminus C^k_{tail},C^k_{tail},{C^k_i\smallsetminus C^k_{tail}}\ldots)=(C^k_0\ldots).$$
Таким образом, поскольку $\overrightarrow{Q}=(C^1_k\ldots),$ получаем, что искомое $i$ есть первое отличающееся неполное частное цепных дробей $[C^k_1]$ и $[C^k_0].$ Следовательно выражения $([\overleftarrow{P}] - [\overrightarrow{R}])$ и $([\overrightarrow{Q}] - [\overleftarrow{Q}])$ имеют одинаковый знак, то есть
$$
([\overleftarrow{P}] - [\overrightarrow{R}])([\overrightarrow{Q}] - [\overleftarrow{Q}])>0.
$$
Отсюда по лемме \ref{oms} отражение $\overrightarrow{Q}$ в формуле (\ref{mainotr}) увеличивает континуант, что и требовалось доказать.
\end{proof}
{\itshape Замечание.}
Нетрудно видеть, что в результате замены континуант примет вид:
\begin{equation}
\begin{split}
\langle\overbrace{\ldots,C_1^k,\underbrace{C^k_0,\ldots, C^k_0,}_{m-k}}^{\overrightarrow{P}}\overbrace{C^k_0,\ldots, C^k_1\smallsetminus C^k_{tail}}^{\overleftarrow{Q}},\overbrace{ C^k_{tail},\underbrace{C^k_0,\ldots, C^k_0}_{m+t-1}\ldots}^{\overrightarrow{R}}\rangle=\\=\langle\overbrace{\ldots,C_1^k,\underbrace{C^k_0,\ldots, C^k_0,}_{m-k}}^{\overrightarrow{P}}\overbrace{C^k_0,\ldots, C^k_1}^{\overleftarrow{Q}},\overbrace{\underbrace{C^k_0,\ldots, C^k_0,}_{m+t-1}\ldots}^{\overrightarrow{R}}\rangle
\end{split}
\end{equation}
и будет также иметь блоковую структуру $k-$го уровня.
\begin{foll}
\label{ozamene}
Если континуант состоит из блоков $C^k_0$ и $C^k_1$ т.е. имеет вид:
\begin{equation}
\label{81}
\langle X, C^k_0, Y, C^k_1, Z\rangle,
\end{equation}
то существует замена отражением, в результате которой континуант примет вид:
$$
\langle X, C^k_1, Y', C^k_0, Z\rangle
$$
и будет также иметь блоковую структуру $k-$го уровня.
\end{foll}
\begin{proof}
Пусть континуант имеет вид (\ref{81}) Произведем следующее отражение $Q$:
$$
\langle X, \overbrace{C^k_0, Y, C^k_1\smallsetminus C^k_{tail}}^{\overrightarrow{Q}},C^k_{tail}, Z\rangle.
$$
Из вышедоказанного следует, что континуант в результате отражения примет вид:
$$
\langle X, \overbrace{C^k_1, Y', C^k_0\smallsetminus C^k_{tail}}^{\overleftarrow{Q}},C^k_{tail}, Z\rangle=\langle X, C^k_1, Y', C^k_0, Z\rangle,
$$
где $Y'$ состоит из блоков $C^k_i.$ Что и требовалось доказать.
\end{proof}
Докажем теперь теорему \ref{maxs}.
\begin{proof} Докажем утверждение \textbf{(i)}.\\
Если континуант не имеет блоковой структуры $k+1-$го уровня, то это означает одну из трех возможных ситуаций:\\
1)  В континуанте встречаются последовательности неполных частных $(C_1^k,\underbrace{C^k_0,\ldots, C^k_0}_{m-k},C_1^k)$ и $(\underbrace{C^k_0,\ldots, C^k_0}_{m+t})$ для каких-то натуральных $k$ и $t.$ В этом случае, как показано в лемме \ref{simq}, существует абсолютно увеличивающая замена отражением. Это означает, что произвольный континуант $\langle A\rangle$ можно при помощи увеличивающих отражений перевести в некоторый континуант $\langle A'\rangle$, в котором не реализуется ситуация 1)\\
2)Континуант начинается с блока $C^1_k$ или заканчивается на блок $C^k_0.$\\
3)Континуант имеет вид
$$
\langle\underbrace{C^k_0,\ldots, C^k_0}_{m-k},\ldots, \underbrace{C^k_0,\ldots, C^k_0}_{m+t},\ldots\rangle~\text{или}~
\langle\ldots,\underbrace{C^k_0,\ldots, C^k_0}_{m+t},\ldots, \underbrace{C^k_0,\ldots, C^k_0}_{m-k}\rangle
$$
для каких-то  натуральных $k$ и $t.$

Разберем 2) и 3). Покажем, что в этих случаях $\langle A'\rangle$ можно при помощи отражений перевести в континуант, имеющий блоковую структуру $k+1-$го уровня. Обозначим за $len(k)$ длину (т.е. количество неполных частных) блока $C^k_0$. Очевидно, $len(k)\geqslant2^{k+1}.$

2)~Пусть континуант $\langle A\rangle$ начинается с блока $C^k_1$.  Выберем в $\langle A\rangle$ произвольный блок $C^k_0$.
$$
\langle A\rangle=\langle C^k_1,\ldots, C^k_0,\ldots \rangle.
$$
Тогда по следствию \ref{ozamene} существует отражение $Q$
$$
\langle\overbrace{C^k_1, \ldots, C^k_0\smallsetminus C^k_{tail}}^{\overrightarrow{Q}},C^k_{tail},\ldots\rangle,
$$
превращающее континуант в 
$$
\langle A'\rangle=\langle C^k_0,\ldots, C^k_1,\ldots\rangle.
$$
Данное отражение может быть уменьшающим. Оценим, во сколько раз оно может уменьшить $\langle A\rangle.$ Из формулы (\ref{comparecont}) следует, что:
$$
\left|\frac{\langle A\rangle-\langle A'\rangle}{\langle A\rangle}\right|<|[\overleftarrow{Q}]-[\overrightarrow{Q}]|.
$$
Оценим по модулю разность $[\overleftarrow{Q}]-[\overrightarrow{Q}]$. Как уже было показано, $[\overleftarrow{Q}]$ имеет вид:
$$
[C^k_0\ldots]=[C^{k-1}_b,\ldots, C^{k-1}_b,C^{k-1}_s\ldots].
$$
Аналогично
$$
[\overrightarrow{Q}]=[C^k_1\ldots]=[C^{k-1}_b,\ldots, C^{k-1}_b,C^{k-1}_s\ldots].
$$
Следовательно, первые $len(k-1)$ неполных частных цепных дробей $[\overleftarrow{Q}]$ и $[\overrightarrow{Q}]$ совпадают, а значит
$$
|[\overleftarrow{Q}]-[\overrightarrow{Q}]|\leqslant\frac1{2^{len(k-1)}}\leqslant\frac1{2^{2^{k}}}.
$$
То есть:
$$
\frac{\langle A'\rangle}{\langle A\rangle}>1-\frac1{2^{2^{k}}}.
$$
Аналогичным преобразованием отражения можно добиться, чтобы континуант заканчивался на блок $C^k_1$. Тогда $\langle A'\rangle$ имеет вид:
$$
\langle A'\rangle=\langle \underbrace{C^k_0, \ldots, C^k_0}_{l_0},C^k_1,\underbrace{C^k_0, \ldots, C^k_0}_{l_1},C^k_1, \ldots, C^k_1\underbrace{C^k_0, \ldots, C^k_0}_{l_{d-1}},C^k_1,\underbrace{C^k_0, \ldots, C^k_0}_{l_d},C^k_1\rangle.
$$
3)~Описанными выше увеличивающими преобразованиями (\ref{mainotr}) можно добиться, чтобы все $l_i$ кроме $l_0$ отличались не более чем на $1$, а ${l_0\leqslant \max\limits_{1\leqslant i\leqslant d}l_i.}$ Если $\max\limits_{1\leqslant i\leqslant d}l_i-l_0\leqslant 1,$ то  $\langle A'\rangle$ имеет блоковую структуру $k+1-$го уровня, и индуктивный переход выполнен. В противном случае существует $i$ такое, что $l_i-l_0>1.$ 
Рассмотрим отражение $Q$:
$$
\langle\underbrace{C^k_0, \ldots, C^k_0}_{l_0}\overbrace{C^k_1,\underbrace{C^k_0, \ldots, C^k_0}_{l_1},\ldots,\underbrace{C^k_0, \ldots, C^k_0}_{l_{i-1}}C^k_1\underbrace{C^k_0, \ldots, C^k_0\smallsetminus C^k_{tail}}_{l_i-l_0}}^{\overrightarrow{Q}},C^k_{tail},\underbrace{C^k_0, \ldots, C^k_0}_{l_0}C^k_1\ldots\rangle.
$$
Данное преобразование, аналогично, уменьшает континуант не более, чем в $1+\frac1{2^{2^k-1}}$ раз.  В результате отражения получим континуант
$$
\langle \underbrace{C^k_0, \ldots, C^k_0}_{l_i},C^k_1,\underbrace{C^k_0, \ldots, C^k_0}_{l_{i-1}},C^k_1, \ldots, C^k_1\underbrace{C^k_0, \ldots, C^k_0}_{l_1},C^k_1,\underbrace{C^k_0, \ldots, C^k_0}_{l_0},C^k_1,\ldots\rangle,
$$
который увеличивающими преобразованиями (\ref{mainotr}) приводится к виду
$$
\langle A''\rangle=\langle \underbrace{C^k_0, \ldots, C^k_0}_{\overline{l_0}},C^k_1,\underbrace{C^k_0, \ldots, C^k_0}_{\overline{l_1}},C^k_1, \ldots, C^k_1\underbrace{C^k_0, \ldots, C^k_0}_{\overline{l_{d-1}}},C^k_1,\underbrace{C^k_0, \ldots, C^k_0}_{\overline{l_d}},C^k_1\rangle,
$$
где $\overline{l_0}=l_i,$ а все $\overline{l_j}$ отличаются друг от друга не более, чем на $1$, $|\overline{l_i}-\overline{l_j}|\leqslant1$. Таким образом, на каждом шаге уменьшающие преобразования уменьшают континуант не более, чем в $\left(1+\frac1{2^{2^k-1}}\right)^3$ раз. Следовательно, в результате $k$ шагов континуант под действием уменьшающихся преобразований уменьшится  суммарно не более, чем в $\prod\limits_{i=1}^{\infty}\left(1+\frac1{2^{2^i-1}}\right)^3<8$ раз.

Докажем утверждение \textbf{(ii)}.\\
Применяя $k$ раз утверждение \textbf{(i)} можно произвольный континуант $\langle A\rangle$ из $M_3(n, S_n)$ увеличивающими преобразованиями отражения перевести в континуант $\langle A'_m\rangle$, имеющий блоковую структуру $k+1-$го уровня. Поскольку данная структура по условию леммы является вырожденной, конец цепочки преобразований не зависит от выбора начального континуанта в $M_3(n, S_n)$. А значит, поскольку под действием уменьшающих преобразований континуант уменьшится  суммарно не более, чем в $8$ раз, по лемме \ref{minalgor} $\langle A'_m\rangle$ отличается от максимума по множеству $M_3(n, S_n)$ не более, чем в $8$ раз.

Таким образом, теорема доказана полностью.
\end{proof}
\begin{foll}
\label{levb}
Если $\langle A_m\rangle=\max(M_3(n,S_n)),$ то\\
\textbf{(i)}~Cуществует такое натуральное $k,$ что $\langle A_m\rangle$ не более, чем в $8$ раз отличается от некоторого алгоритмически построимого континуанта  $\langle A'_m\rangle$ из $M_3(n,S_n),$ имеющего вырожденную блоковую структуру $k-$го уровня.\\
\textbf{(ii)}~$\langle \underbrace{A'_m, A'_m,\ldots, A'_m}_{l}\rangle$ не более, чем в $8$ раз отличается от $\max(M_3(ln,lS_n))$
\end{foll}
\begin{proof}
\quad\\
\textbf{(i)}~Применим к  $\langle A_m\rangle$ рекурсивный алгоритм поиска максимума. В результате данного алгоритма он перейдет в $\langle A'_m\rangle$ под действием некоторой цепочки отражений. Как следует из теоремы \ref{maxs} $\langle A'_m\rangle$  не более, чем в $8$ раз отличается от $\langle A_m\rangle,$ что и требовалось доказать.\\
\textbf{(ii)}~Поскольку $\langle A'_m\rangle$ имеет вырожденную блоковую структуру $k-$го уровня для некоторого натурального $k,$ то $\langle \underbrace{A'_m, A'_m,\ldots, A'_m}_{l}\rangle$ также имеет вырожденную блоковую структуру $k-$го уровня. Для завершеия доказательства следствия остается только воспользоваться утверждением  \textbf{(ii)} теоремы \ref{maxs}.
\end{proof}
Таким образом, задача на поиск асимптотики  максимума решена, предъявлен алгоритм, позволяющий прийти к нему за конечное число шагов. Назовем результат работы алгоритма \textit{асимптотическим максимумом} по множеству $M(n,S_n)$ и обозначим его $\max_a(n, S_n).$ Из теорем \ref{edvar}, \ref{thto3} и \ref{maxs} следует, что существует некоторая не зависящая от $n$ и $S_n$ константа $c_0,$ что выполнено неравенство:
$$
1\leqslant\frac{\max(n, S_n)}{\max_a(n, S_n)}<c_0.
$$
Докажем лемму об асимптотике периодического континуанта.
\begin{lem}
\label{count}
\begin{equation}
\langle\underbrace{A,\ldots,A}_{n}\rangle\sim c(\langle A\rangle+\langle  A^-\rangle\lambda)^n,
\end{equation}
где $\lambda=[\overline{A}] -$ квадратичная иррациональность, а $c$ - некоторая, зависящая от $A,$ но не зависящая от $n$ константа.
\end{lem}
\begin{proof}
Пользуясь правилом раскрытия континуантов (\ref{knu}), получаем:
$$
\langle\underbrace{A,\ldots,A}_{n}\rangle=\langle A \rangle\langle\underbrace{A,\ldots,A}_{n-1}\rangle+\langle A^- \rangle\langle A_-\underbrace{A,\ldots,A}_{n-2}\rangle=\langle\underbrace{A,\ldots,A}_{n-1}\rangle(\langle A \rangle+\langle A^- \rangle[\underbrace{A,\ldots,A}_{n-1}]).
$$
Следовательно:
$$
\langle\underbrace{A,\ldots, A}_{n}\rangle=\langle A \rangle\prod\limits_{i=2}^n(\langle A \rangle+\langle A^- \rangle[\underbrace{A,\ldots,A}_{i-1}]).
$$
Докажем, что отношение этого выражения к $(\langle A\rangle+\langle  A^-\rangle\lambda)^n$ стремится к константе. Рассмотрим отношение
$$
\prod\limits_{i=2}^{\infty}\frac{\langle A\rangle+\langle  A^-\rangle\lambda}{\langle A \rangle+\langle A^- \rangle[\underbrace{A,\ldots,A}_{i-1}]}=\prod\limits_{i=2}^{\infty}\frac{\langle A\rangle+\langle  A^-\rangle([\overbrace{A,\ldots, A}^{i-1}]-r_n)}{\langle A \rangle+\langle A^- \rangle[\underbrace{A,\ldots,A}_{i-1}]},
$$
где $r_n=[\underbrace{A,\ldots,A}_{i-1}]-\lambda.$\\ Поскольку $|r_n|<\frac1{2^n},$ получаем:
$$
\prod\limits_{i=2}^{\infty}\frac{\langle A\rangle+\langle  A^-\rangle([\overbrace{A,\ldots,A}^{i-1}]-r_n)}{\langle A \rangle+\langle A^- \rangle[\underbrace{A,\ldots,A}_{i-1}]}=\prod\limits_{i=2}^{\infty}(1-\frac{\langle A^-\rangle r_n}{\langle A \rangle+\langle A^- \rangle[\underbrace{A,\ldots, A}_{i-1}]}).
$$
По известному свойству это бесконечное произведение сходится тогда и только тогда, когда абсолютно сходится ряд
$$
\sum\limits_{i=2}^{\infty}r_n\frac{\langle A^-\rangle}{\langle A \rangle+\langle A^- \rangle[\underbrace{A,\ldots,A}_{i-1}]},
$$
что выполнено, поскольку выражение
$$
\frac{\langle A^-\rangle}{\langle A \rangle+\langle A^- \rangle[\underbrace{A,\ldots,A}_{i-1}]}
$$
ограниченно положительными константами. А следовательно, бесконечное произведение имеет предел, лемма доказана.
\end{proof}
\section{Доказательства теорем} 
Доказательство теоремы \ref{res1}.
\begin{proof}
Пользуясь леммой \ref{lemle}, получаем
$$
\frac{g_{\varphi^{-1}}(x+\delta)-g_{\varphi^{-1}}(x)}{\delta}\geqslant \frac{q_tq_{t-1}}{\varphi^{S^{\varphi}_t(x)-5}}\geqslant \frac{\varphi^{2t-1}}{\varphi^{\frac{t\varkappa_{inf}(x)}2}},
$$
что стремится к $+\infty$ при $\varkappa_{inf}(x)<4.$ Первая часть теоремы доказана.

Для доказательства второй части выберем следующие параметры:
$$\frac12<\alpha\in\mathbb{Q}<1,~~ 0<\varepsilon<\alpha-\frac12,~~
m\in\mathbb{N}:~3m\leqslant\varphi^{m\varepsilon}~\text{и при этом}~m\alpha\in\mathbb{N}.$$
Рассмотрим теперь квадратичную иррациональность $x=[\overline{\underbrace{1,\ldots,1}_{2m-1},\alpha m+1}].$ Для нее
$$
\varkappa_{inf}(x)=1+2\frac{m-1+\alpha m+1}{m}=3+2\alpha.
$$
Поскольку все неполные частные $x$ ограничены, то из леммы \ref{lemge} следет, что для некоторого $t=2Nm$ выполнено
\begin{equation}
\label{le2}
\frac{g_{\varphi^{-1}}(x+\delta)-g_{\varphi^{-1}}(x)}{\delta}\leqslant  C\frac{q^2_{t}}{\varphi^{a_1+2a_2+\ldots+a_{t-1}+2a_{t}}},
\end{equation}
Оценим сверху $q_t$. Поскольку
$$
\langle A,B\rangle\le2\langle A\rangle\langle B\rangle~\text{и}~\langle\underbrace{1,\ldots,1}_{2m-1}\rangle<\varphi^{2m-1},
$$
получаем:
\begin{multline}
\label{o3}
q_t=\langle\overbrace{\underbrace{1,\ldots,1}_{2m-1},\alpha m,\ldots,\underbrace{1,\ldots,1}_{2m-1},\alpha m}^{N~\text{повторений}}\rangle\leqslant2^N{\langle\underbrace{1,\ldots,1}_{2m-1},\alpha m\rangle}^{N}\leqslant\\ \leqslant{(4\alpha m)}^N\varphi^{2mN-N}\leqslant{\biggl(3\alpha m\varphi^{2m}\biggr)}^N\leqslant \varphi^{{(m(2+\varepsilon)}^N)}=\varphi^{mN(2+\varepsilon)}.
\end{multline}
Таким образом, объединяя оценки (\ref{le2}) и (\ref{o3}), имеем:
$$
\frac{g_{\varphi^{-1}}(x+\delta)-g_{\varphi^{-1}}(x)}{\delta}\leqslant  C\frac{\varphi^{mN(4+2\varepsilon)}}{\varphi^{mN(3+2\alpha)}}=C\varphi^{mN(1+2\varepsilon-2\alpha)},
$$
что стремится к $0$ при $N\to \infty$.\\
Таким образом, в силу того, что $\alpha$ можно выбрать сколь угодно близким к $\frac12,$ второе утверждение теоремы доказано.
\end{proof}
Доказательство теоремы \ref{res2}.
\begin{proof}
В доказательстве данного утверждения мы докажем существование $\varkappa_2$ и предъявим алгоритм, позволяющий получить его с любой точностью, а также оценим сверху скорость его сходимости.

Во-первых отметим, что $g'_{\varphi^{-1}}([\overline{7,4}])=0.$ Действительно, по лемме \ref{count} 
$$
q_{2t}=\langle4,7,\ldots,4,7\rangle\asymp(29+4[\overline{7,4}])^t<30^t.
$$
А поскольку 
$$\varphi^{a_1+2a_2+\ldots+a_{2t-1}+2a_{2t}}=\varphi^{15t}>1000^t,$$
получаем, что
$$\lim\limits_{t\to\infty}\frac{q^2_{2t}}{\varphi^{a_1+2a_2+\ldots+a_{2t-1}+2a_{2t}}}=0.$$
Следовательно, по лемме \ref{lemle} производная в точке $[\overline{7,4}]$ существует и равна $0$.\\
Далее, пусть $x=[a_1,\ldots, a_t,\ldots]$ - иррациональное число и $\varkappa_{inf}(x)>15$. Тогда по лемме \ref{lemle}
\begin{equation}
\label{tolemle}
\frac{g_{\varphi^{-1}}(x+\delta)-g_{\varphi^{-1}}(x)}{\delta}\leqslant \frac{q^2_{t}}{\varphi^{S^{\varphi}_t(x)-5}}
\end{equation}
Напомним, что
$\max(M^{\varphi}(t, S_t))$ - максимум по всем континуантам длины $t$ с $S^{\varphi}_t(x)=S_t.$ Очевидно, что
$$
\frac{g_{\varphi^{-1}}(x+\delta)-g_{\varphi^{-1}}(x)}{\delta}\leqslant \frac{q^2_{t}}{\varphi^{S^{\varphi}_t(x)-5}}\leqslant C\frac{\max^2(M^{\varphi}(t, S_t))}{\varphi^{S^{\varphi}_t(x)}}<C_1\frac{\max_a^2(M^{\varphi}(t, S_t))}{\varphi^{S^{\varphi}_t(x)}}.
$$
Поскольку по теореме \ref{thto3}
$\langle A_{max}\rangle=\max_a(M^{\varphi}(t, S_t))\in M_3(t, S_t),$ а ${\varkappa_{inf}(x)>15,}$ то по следствию \ref{to47} любое неполное частное $\langle A_{max}\rangle$ больше либо равно любого соответствующего неполного частного континуанта\\ $\langle B\rangle=\langle\underbrace{7,4,\ldots,7,4}_{\text{t/2 пар}}\rangle.$ Отсюда следует, что $\langle B\rangle$ можно превратить в $\langle A\rangle$ увеличением некоторых неполных частных на $1$ (возможно несколько раз). Однако несложно убедиться, что любая такая замена уменьшает дробь в правой части формулы (\ref{tolemle}). Следовательно $\varkappa_2$ существует и меньше $15$.

Аналогично несложно показать, что $g'_{\varphi^{-1}}([\overline{7,3}])=+\infty$, а значит, ${\varkappa_2>13}$
Тогда из теоремы \ref{thto3} следует, что $\max_a(M^{\varphi}(t, S_t))$ для всех $x$, для которых выполнено
$$
13<\frac{S^{\varphi}_t(x)}t<15
$$
достигается на множестве  состоящем из континуантов $\langle A\rangle$ таких, что $N(A)=(\{7\},\{3,4\})$. Обозачим это множество $C^7_{3,4}$.
Сформулируем следующий простой принцип:\\
Пусть $x=[\overline{A}], y=[\overline{B}]$ - периодические цепные дроби,
${A,B\in C^7_{3,4}},\\
 l(A)=t_1, l(B)=t_2$, причем:
\begin{equation}
\label{aib}
\langle A\rangle=\max_a(M^{\varphi}(t_1,S^{\varphi}_{t_1}(A))),~\langle B\rangle=\max_a(M^{\varphi}(t_2, S^{\varphi}_{t_2}(B)))
\end{equation}

И пусть ${g'_{\varphi^{-1}}(x)=\infty},~{g'_{\varphi^{-1}}(y)=0}.$
Тогда $$\varkappa_{inf}(x)\leqslant \varkappa_2\leqslant\varkappa_{inf}(y).$$
Действительно, пусть $\varkappa_2<\varkappa_{inf}(x)$. Это противоречит определению $\varkappa_2$, поскольку $\varkappa_{inf}(x)>\varkappa_2,$ но
$g'_{\varphi^{-1}}(x)=\infty.$

Если же $\varkappa_2>\varkappa_{inf}(y),$ то это означает, что $\exists z: \varkappa_{inf}(y)<\varkappa_{inf}(z)<\varkappa_2$ и при этом $g'_{\varphi^{-1}}(z)=\infty.$ Пусть $z=[c_1\ldots c_t\ldots],$  тогда
$$
\frac{g_{\varphi^{-1}}(z+\delta)-g_{\varphi^{-1}}(z)}{\delta}\leqslant \frac{q^2_{t}(z)}{\varphi^{S^{\varphi}_t(z)-5}}\leqslant \frac{\max^2(M^{\varphi}(t,~S^{\varphi}_t(z)))}{\varphi^{S^{\varphi}_t(z)-5}}.
$$
\begin{lem}
Функция
\begin{equation}
\label{ftstf}
f(S^{\varphi}_t(z))=\frac{{\max_a}^2(M^{\varphi}(t,~S^{\varphi}_t(z)))}{\varphi^{S^{\varphi}_t(z)}}
\end{equation}
при достаточно большом $t$ убывает с ростом $S^{\varphi}_t(z)$ при $13<\frac{S^{\varphi}_t(z)}{t}<15.$ 
\end{lem}
\begin{proof}
Действительно, если
$$S^{\varphi}_t(z_1)>S^{\varphi}_t(z_2)~\text{и}~
\langle A_1\rangle={\max}_a(M^{\varphi}(t,~S^{\varphi}_t(z_1))),
\langle A_2\rangle={\max}_a(M^{\varphi}(t,~S^{\varphi}_t(z_2))),$$
то по теореме \ref{thto3} ${A_1,A_2\in C^7_{3,4}}.$ Докажем неравенство
\begin{equation}
\label{43frac}
\frac{\langle A_2\rangle}{\langle A_1\rangle}<c_0\biggl(\frac43\biggr)^{(S^{\varphi}_t(z_2)-S^{\varphi}_t(z_1))/2}.
\end{equation}
Возьмем любые $\frac{S^{\varphi}_t(z_2)-S^{\varphi}_t(z_1)}2$ неполных частных $\langle A_2\rangle,$ равных $4$ и заменим их на $3$. Так как каждая такая замена уменьшает континуант не более, чем в $\frac43$ раза, имеем оценку
$$\frac{\langle A_2\rangle}{\langle A'_2\rangle}<\biggl(\frac43\biggr)^{(S^{\varphi}_t(z_2)-S^{\varphi}_t(z_1))/2}.$$
А поскольку
$S^{\varphi}_t(A'_2)=S^{\varphi}_t(A_1), \text{то}~\langle A'_2\rangle<c_0{\max}_a(M^{\varphi}(t,~S^{\varphi}_t(z_1)))=c_0\langle A_1\rangle$, что доказывает неравенство (\ref{43frac}).\\
Далее,
\begin{equation}
\label{abphi}
\frac{\varphi^{(S^{\varphi}_t(z_2))}}{\varphi^{(S^{\varphi}_t(z_1))}}=\varphi^{(S^{\varphi}_t(z_2)-S^{\varphi}_t(z_1))}>\biggl(\frac32\biggr)^{(S^{\varphi}_t(z_2)-S^{\varphi}_t(z_1))},
\end{equation}
откуда, подставляя (\ref{43frac}) и (\ref{abphi}) в (\ref{ftstf}), получаем убывание функции $f(S^{\varphi}_t(z))$ при достаточно большом $t.$
\end{proof}
Таким образом,
$$
\frac{\max_a^2(M(t, S^{\varphi}_t(z)))}{\varphi^{S^{\varphi}_t(z)}}\leqslant \frac{\max_a^2(t, S^{\varphi}_t(y)}{\varphi^{S^{\varphi}_t(y)}}\leqslant  C_1\frac{q^2_t(y))}{\varphi^{S^{\varphi}_t(y)}},
$$
где последнее неравенство выполнено по лемме \ref{simlem}. А из того, что
$$
{\lim\limits_{t\to\infty}\frac{q^2_t(y))}{\varphi^{S^{\varphi}_t(y)}}=0},
$$
мы  получаем противоречие с тем, что $g'_{\varphi^{-1}}(z)=+\infty.$

Доказанный принцип позволяет найти $\varkappa_2$ с любой точностью. Вычисления показывают, что для $$x=[\overline{\underbrace{7,3,\ldots,7,3}_{\text{37 пар}}7,4}]~\quad g'_{\varphi^{-1}}(x)=0,$$
а для $$y=[\overline{\underbrace{7,3,\ldots,7,3}_{\text{38 пар}}7,4}]~\quad g'_{\varphi^{-1}}(y){=\infty}.$$
Следовательно 
$$
13.0513\asymp13\frac2{39}<\varkappa_2<13\frac2{38}\asymp13.0526.
$$
Проводя итерации алгоритма с блоками все более высокого уровня можно сосчитать $\varkappa_2$ с любой требуемой точностью. Оценим скорость сходимости алгоритма.

Прежде всего рассмотрим для введенных в (\ref{aib}) континуантов $\langle A\rangle$ и $\langle B\rangle$  континуанты $\langle A'\rangle\in M_3(t_1,S^{\varphi}_{t_1}(A))$ и $\langle B'\rangle\in M_3(t_2,S^{\varphi}_{t_2}(B)),$ являющиеся асимптотическими максимумами по соответствующим множествам. Ввиду теоремы \ref{maxs} они имеют вырожденную блоковую структуру. Напомним, что $t_1$ и $t_2$ - это длины $\langle A\rangle$ и $\langle B\rangle$ cоответственно. В силу теоремы \ref{maxs} и следствия \ref{levb} для любого натурального $i$ выполнены оценки
$$
1\leqslant\frac{\langle \overbrace{A,\ldots, A}^{i}\rangle}{\langle \underbrace{A',\ldots,A'}_{i}\rangle}<8,\quad1\leqslant\frac{\langle \overbrace{B,\ldots,B}^{i}\rangle}{\langle \underbrace{B',\ldots,B'}_{i}\rangle}<8.
$$
Следовательно, 
$$
g'_{\varphi^{-1}}([\overline{A}])=g'_{\varphi^{-1}}([\overline{A'}])=\infty,\quad g'_{\varphi^{-1}}([\overline{B}])=g'_{\varphi^{-1}}([\overline{B'}])=0.
$$
Без ограничения общности будем считать, что $t_1=t_2,$ поскольку в противном случае мы можем перейти к рассмотрению цепных дробей 
$$
[\overline{\underbrace{A',\ldots,A'}_{t_2}}]=[\overline{A'}]~~\text{и}~~[\overline{\underbrace{B',\ldots,B'}_{t_1}}]=[\overline{B'}],
$$
имеющих одинаковую длину. Рассмотрим континуант $\langle A',B'\rangle,$ обозначим его $\langle C'\rangle.$ Очевидно, что $l(C')=2t_1,$ a $S^{\varphi}_{2t_1}(C')=S^{\varphi}_{t_1}(A')+S^{\varphi}_{t_1}(B').$ Это означает, что
$$
\varkappa_{inf}([\overline{C'}])=\frac{\varkappa_{inf}([\overline{A'}])+\varkappa_{inf}([\overline{B'}])}2.
$$
Обозначим $\langle C\rangle=\max(M_3(2t_1, S^{\varphi}_{2t_1}(C'))).$ Найдем, чему равна производная в точке $[\overline{C}].$ Если она равна $0,$ то по сформулированному выше принципу максимума 
$$
\varkappa_{inf}([\overline{A'}])<\varkappa_2<\varkappa_{inf}([\overline{C'}]).
$$
Если, напротив, $g'_{\varphi^{-1}}([\overline{C'}])=\infty,$ то, аналогично,
$$
\varkappa_{inf}([\overline{C'}])<\varkappa_2<\varkappa_{inf}([\overline{B'}]).
$$
Таким образом, за один шаг алгоритма отрезок, на котором лежит $\varkappa_2$ уменьшается в $2$ раза. Следовательно, для того, чтобы найти $\varkappa_2$ c точностью $\varepsilon$ требуется не более $\log{\varepsilon^{-1}}$ шагов алгоритма. 
\end{proof}
Доказательство теоремы \ref{res3}.
\begin{proof}
Заметим, что $g'_{\varphi^{-1}}([\overline{1,2}])=+\infty,$ поскольку по лемме~\ref{count} $$q^2_{2t}=(3+(\sqrt{3}-1))^{2t}>13^t$$ и знаменатель дроби из формулы (\ref{mainlemge}) равен $\varphi^{S^{\varphi}_{2t}(x)}=\varphi^{5t}<12^t$, а значит, по данной лемме производная существет и равна $+\infty.$

Пусть теперь $x$-произвольное число с неполными частными, ограниченными $2.$ Рассматривая также для него дробь 
\begin{equation}
\label{finfrac}
\frac{q^2_{2t}(x)}{\varphi^{S^{\varphi}_{2t}(x)}},
\end{equation}
заметим, что континуант $q_{2t}(x)$ получен из континуанта $$\langle\underbrace{1,2,\ldots,1,2}_{\text{t пар}}\rangle$$ заменой некоторых единиц на двойки для нечетных неполных частных и заменой двоек на единицы для четных неполных частных. Нетрудно видеть, что обе замены увеличивают дробь (\ref{finfrac}). Действительно, замена $1$ на $2$ по формуле (\ref{kor}) увеличивает континуант не менее, чем в $\frac43$ раза, а значит числитель увеличится как минимум, в $\frac{16}9$ раза в то время, как знаменатель увеличится в $\varphi$ раз, следовательно дробь (\ref{finfrac}) возрастет. Аналогично, поскольку все цепные дроби с неполными частными из нашего континуанта меньше $\frac13$, замена $2$ на $1$ уменьшит континуант не более, чем в $$\frac{2+\frac23}{1+\frac23}=\frac85<\varphi$$ раз, а следовательно дробь (\ref{finfrac}) также увеличится.

Для доказательства последней части утверждения заметим, что\\  $g'_{\varphi^{-1}}([\overline{1,3}])=0,$ что проверяется аналогично. Теорема доказана полностью.
\end{proof}
Доказательство теоремы \ref{res4}.
\begin{proof}
Докажем, что если $x=[\overline{\overrightarrow{A}}]$ - периодическая цепная дробь, причем длина $A$ четная, а $y=[\overline{\overleftarrow{A}}]$, то $g'_{\varphi^{-1}}(x)=g'_{\tau}(y).$ Действительно, поскольку для любого $m\in\mathbb{N}$ выполнено
$$
\langle\underbrace{\overrightarrow{A},\ldots, \overrightarrow{A}}_{m}\rangle=\langle\underbrace{\overleftarrow{A},\ldots, \overleftarrow{A}}_{m}\rangle
$$
то для любого натурального $t~ q_t(x)\asymp q_t(y).$ Аналогично, поскольку для любого $m\in\mathbb{N}$ выполнено
$$
S^{\varphi}(\underbrace{\overrightarrow{A},\ldots,\overrightarrow{A}}_{m})=S^{\tau}(\underbrace{\overleftarrow{A},\ldots, \overleftarrow{A}}_{m})
$$
то для любого натурального $t~ S^{\varphi}_t(x)\asymp S^{\varphi}_t(y).$ Теперь наше утверждение автоматически следует из лемм части 4.

Отсюда следует, что константы в теоремах  \ref{res1}, \ref{res2} и \ref{res3}  для $g_{\tau}(x)$ те же самые и утверждения теорем  \ref{res1}, \ref{res2} и \ref{res3} для них доказываются аналогично.
\end{proof}

\end{document}